\documentclass[12pt,leqno]{article}
\usepackage{amsmath, amsthm, amsfonts, amssymb}
\usepackage{mathrsfs}
\usepackage{color}
\usepackage{xcolor}
\usepackage{hyperref}

\newtheorem{thm}{Theorem}[section]
\newtheorem{cor}[thm]{Corollary}
\newtheorem{lem}[thm]{Lemma}
\newtheorem{prp}[thm]{Proposition}

\newtheorem{exa}[thm]{Example}
\theoremstyle{definition}
\newtheorem{rem}[thm]{Remark}

\numberwithin{equation}{section}

\allowdisplaybreaks

\newcommand{\scr}[1]{\mathscr #1}

\def\eins{\boldsymbol 1}
\def\RR{\mathbb R}
\def\NN{\mathbb N}
\def\WW{\mathbb W}
\def\HH{\mathbb H}

\def\DDp{{\scr D}_1}

\def\E{\mathbb E}
\def\D{\scr D}
\def\B{\scr B}
\def\F{\scr F}
\def\P{\scr P}
\def\N{\scr N}

\def\pac{\scr P^\textnormal{ac}}
\def\ppac{\scr P_p^\textnormal{ac}}
\def\twopac{\scr P_2^\textnormal{ac}}

\def\d{\textnormal {d}}

\def\loc{\textnormal {loc}}

\def\det{\textnormal{det}}
\def\id{\textnormal{id}}
\def\supp{\textnormal{supp}}

\def\tr{\textnormal {T}}
\def\lip{\textnormal {Lip}}

\def\lam{\lambda}
\def\Lam{\Lambda}

\def\na{\nabla}
\def\la{\langle}
\def\ra{\rangle}

\def\R{\mathbb R}  \def\ff{\frac} \def\B{\mathbf
	B} \def\W{\mathbb W}
\def\N{\mathbb N}  

  \def\vv{\varepsilon} \def\rr{\rho}
\def\<{\langle} \def\>{\rangle} \def\GG{\Gamma} 
\def\nn{\nabla} \def\pp{\partial}
\def\d{\text{\rm{d}}}  \def\aa{\alpha} \def\D{\scr D}
 
\def\beg{\begin} \def\beq{\begin{equation}}  \def\F{\scr F}
	 
	\def\e{\text{\rm{e}}}    
	 
	 \def\P{\mathbb P}

	  \def\ll{\lambda}

	  \def\LL{\Lambda}
	 \def\B{\scr B}  
	\def\to{\rightarrow}
	\def\EE{\scr E}
	 
	\def\BB{\scr B}
	\def\R{\mathbb R}  \def\ff{\frac}  \def\B{\mathbf
		B}
	\def\N{\mathbb N}  
	  \def\vv{\varepsilon} \def\rr{\rho}
	\def\<{\langle} \def\>{\rangle} \def\GG{\Gamma} 
	\def\nn{\nabla} \def\pp{\partial} 
	\def\d{\text{\rm{d}}}  \def\aa{\alpha} \def\D{\scr D}
	 
	\def\beg{\begin} \def\beq{\begin{equation}}  \def\F{\scr F}
		 
		\def\e{\text{\rm{e}}}    
		 
		 \def\P{\mathbb P}

		  \def\ll{\lambda}

		  \def\LL{\Lambda}
		 \def\B{\scr B}  
		\def\to{\rightarrow}\def\BB{\mathbb B}
		\def\8{\infty} 
		
		\author{Panpan Ren$^{e)}$, Michael R\"ockner$^{b,c,d)}$, Feng-Yu Wang$^{a)}$ \\and Simon Wittmann$^{f,g,* )}$\\
	\footnotesize{a) Center for Applied Mathematics and KL-AAGDM, Tianjin
		University, Tianjin, China}\\
	\footnotesize{b) Faculty of Mathematics, Bielefeld University, Bielefeld, Germany}\\
	\footnotesize{c) Academy of Mathematics and Systems Science, Chinese Academy of Sciences, Beijing, China}\\
	\footnotesize{d) School of Data Science, The Chinese University of Hongkong, Shenzhen, China}\\
	\footnotesize{e) Department of Mathematics, City University of  Hong Kong, Hong Kong,  China}\\
	\footnotesize{f) Department of Applied Mathematics, The Hong Kong Polytechnic University,  Hong Kong,  China }\\
	\footnotesize{g) Faculty of Mathematics, Wroc{\l}aw University of Science and Technology, Wroc{\l}aw, Poland}\\
	\footnotesize{ $*$) Corresponding author, \textit{simon.wittmann@pwr.edu.pl}}
}

\title{{\bf Stochastic intrinsic gradient flows on the Wasserstein space}\footnote{Supported in part by the National Key R\&D Program of China (2022YFA1006000), NNSFC(12301180),  RGC (21301925),  NSFC/RGC JRS N-CityU165/25,  State Key Lab, and Deutsche Forschungsgemeinschaft (DFG, German Research Foundation), Project-ID 317210226, SFB 1283. Panpan Ren and Simon Wittmann are supported by Research Centre for Nonlinear Analysis at Hong Kong PolyU.} 
}
		\date{}
		
\begin{document}
\maketitle

\begin{abstract}
We construct stochastic gradient flows on the $2$-Wasserstein space $\scr P_2$ over $\mathbb R^d$ for energy functionals of the type $W_F(\rho d x)=\int_{\mathbb R^d}F(x,\rho(x))d x$. The functions $F$ and $\partial_2 F$ are assumed to be locally Lipschitz on $\mathbb R^d\times (0,\infty)$. This includes the relevant examples of $W_F$ as the entropy functional or more generally the Lyapunov function of generalized porous media equations. First we define a class of Gaussian-based measures $\Lambda$ on $\scr P_2$ together with a corresponding class of symmetric Markov processes ${(R_t)}_{t\geq 0}$. Next, using Dirichlet form techniques we perform stochastic quantization for the perturbations of these objects which result from multiplying such a measure $\Lambda$ by a density proportional to $e^{-W_F}$. Finally we show that the intrinsic gradient $DW_F(\mu)$ is defined for $\Lambda$-a.e.~$\mu$ and that the Gaussian-based reference measure $\Lambda$ can be chosen in such way that the distorted process ${(\mu_t)}_{t\geq 0}$ is a martingale solution for the equation $d\mu_t=-DW_F(\mu_t) d t+d R_t$, $t\geq 0$.
\end{abstract}


\section{Introduction}

\subsection{The generalized porous media equation}

  In the pioneering work of \cite{JKO98} the solution to the linear Fokker-Planck-Kolmogorov equation, which is a parabolic partial differential equation describing a time-dependent probability density on $\RR^d$, is shown to run along the steepest descent of its Lyapunov function.
This result is derived by a time-discrete iterative procedure, in which the incremental decrease of the energy function is measured in relation to the $2$-Wasserstein distance. 
A genuinely geometric concept, on the other hand, is first developed in \cite{Otto} to analyse the gradient flow nature of the porous medium equation. It emulates a Riemannian structure on the set of probability densities and is closely related to the notions regarding tangential spaces and gradients used in this article.

In the wake of those discoveries there has been an increased interest in gradient flows on metric spaces, in particular in connection with Optimal Transport. 
The monograph \cite{AGS05} provides a comprehensive survey on this topic. Further contributions expand our understanding 
of gradient flows associated with Markov semigroups and parabolic equations in various non-Euclidean settings. 
Studies include gradient flows of the relative entropy functional regarding probability measures over compact Alexandrov spaces (\cite{O09}), 
semi-concave metric measure spaces with lower Ricci curvature bounds and local angle conditions (\cite{S07}), Riemannian manifolds with lower Ricci curvature bounds \cite{E10},
the Wiener space (\cite{F10}), and other structures (\cite{K18}, \cite{ Ost09}, \cite{OS25}).
Another interesting and ongoing direction of research in this field aims at exploring gradient flow structures behind other types of evolution equations under the use of modified transport distances (see e.g.~\cite{DNS09}, \cite{E14}, \cite{E21}, \cite{GM22}).
For an overview of Optimal Transport and more related topics, applications and examples, we refer to \cite{Fig21} and  \cite{V09}.

In view of the great importance of gradient flows for evolution equations in physics it is a nearby idea to set up an analogous stochastic model. 
This is the purpose of our work.
We solve a martingale problem for gradient dynamics on the space of probability measures $\scr P$ over $\RR^d$.
The drift term points towards the negative direction of the gradient of a physically relevant energy functional $W:\D(W)\to\RR$ with a domain $\D(W)\subset\scr P$,
typically within the set of absolutely continuous measures.
On the Euclidean space, a distorted Brownian motion
\begin{equation}\label{distBM}
	\d X_t=-\nabla W (X_t)\d t+\d B_t,\quad t\geq 0,
\end{equation}
given a suitable energy functional $W:\RR^d\to \RR$,
has an invariant measure with a density proportional to $\e^{-W(x)}\d x$. Moreover, the solution to \eqref{distBM} is the Markov process associated with a gradient-type Dirichlet form 
proportional to $\int_{\RR^d}\langle \na u(x),\na v(x)\rangle \e^{-W(x)}\d x$ for $u,v\in C^1_b(\RR^d)$ (cf.~the early works \cite{AHS,F81}).
An analogous approach to introduce and treat a stochastic differential equation of gradient type on the Wasserstein space over $\RR^d$ is worked out in this article.
This involves making reasonable choices of substitutes for the Lebesgue measure and Brownian motion, as these don't have canonical analogues on the space of probability measures.

In \cite{RS07} a Dirichlet form approach is used to construct a Markov process
with the intention of a canonical diffusion on the set of probability measures $\scr P([0,1])$ over the closed unit interval.
This state space, equipped with the 2-Wasserstein distance, is isometric to $\scr G_0$, the set of right continuous, non-decreasing functions on the unit interval interval with the
topology and differential structure of $L^2([0,1])$.
The entropic measure, which is the stationary distribution of the dynamics, coincides with the law of a normalized Gamma process on $\scr G_0$. 
The Markov process associated to the respective gradient-type Dirichlet form, known as Wasserstein diffusion, solves the martingale problem for its generator.
The latter, in addition to a natural diffusion part, which corresponds to the Laplace operator on $\scr P([0,1])$, contains two drift terms: one originates from the
jumps of the Gamma process and the other is the generator of the heat flow on $[0,1]$ with Neumann boundary conditions.

A related model is the modified massive Arratia flow, proposed and analyzed in \cite{Kon17, Kon14}. It describes coalescing particles on the real line with a variable diffusivity, inversely proportional to their mass. \cite{D22} introduces a process on the 2-Wasserstein space over a closed Riemannian manifold of dimension greater or equal $2$, whose generator is analogous to that of the modified massive Arratia flow. The multi-dimensional version does not exhibit coalescence 
and is the superprocess of countable many independent, massive Brownian particles.
These dynamics and the Wasserstein diffusion share some characteristics, for example the
validity of Varadhan's short time asymptotic, and the expression for the quadratic variation of their martingale part matches the 
geometric structure of \cite{Otto}.

Regarding the quadratic variation of the martingale part, the diffusion process in this article is consistent with the aforementioned examples. However, the drift term is of a completely different type,
as we are interested in a stochastic version of the gradient flow associated with the generalized porous media equation on $\RR^d$.
Therefore, our drift term contains the gradient of a relevant energy functional $W:\D(W)\to\RR$, defined only on a subset of the absolutely continuous measures. 
The invariant measure differs largely
from the entropic measure (\cite{RS07}, \cite{St11}, \cite{St24}) or the Dirichlet-Ferguson measure (\cite{D22}, \cite{Shao11}) regarding its support properties. While those are random measures with almost surely no absolutely continuous part, we consider a weighted version of the Gaussian based random measure of \cite{RW22}. 
In our case, the event that $W$ is differentiable (this in particular requires the random measure to be absolutely continuous) has probability one.
We characterize our diffusion in relation to a noise term ${(R_t)}_{t\geq 0}$ originating from the Gaussian reference measure.
The resulting stochastic differential equation 
\begin{equation}\label{zero}
	\d\mu_t=-DW_t(\mu_t) \d t+\d R_t,\qquad t\geq 0,
\end{equation}
 is interpreted in the sense of a martingale solution for the generator.

The process ${(R_t)}_{t\geq 0}$ in this article has some similarities with the construction in \cite{Del}, where
a reflecting stochastic heat equation on the set of quantile functions of probability measures over $\RR$ is solved by means of a Euler scheme  alternating 
between time-incremental approximation steps for the classical stochastic heat equation and symmetric rearrangements.
On the other hand, our approach
 does not involve a one-to-one identification of the Wasserstein space with a subset of an $L^2$-space, since
 the push-forward  of a stochastic element from a suitable function space into the Wasserstein space is performed on the level of the equilibrium distribution.
Then, Dirichlet form techniques are used  to obtain the corresponding measure-valued Markov process.
 This approach is more versatile, as it  works in the multi-dimensional setting, i.e.~for probability measures over $\RR^d$, $d\geq 1$.
To our best knowledge, stochastic gradient dynamics as in \eqref{zero} for the energy functional of the generalized porous media equation have not been treated in the literature before.

We recall the notion of a gradient for functions on probability measures. 
Let $d\in\NN$ and $\scr P$ denote the set of all Borel probability measures on $\RR^d$.  
Following an approach as in \cite{AKR96, AKR98, R98, D20, BRW21, RRW22, RR23} the differential and the gradient of a function $W:\scr P\to\RR$ at a point $\mu\in\scr P$ may be assigned
by looking at the family of curves $\mu_{\varphi,\varepsilon}:=\mu\circ (\textnormal{id}+\varepsilon\varphi)^{-1}\in\scr P$, $\varphi\in C_b^1(\RR^d,\RR^d)$, parameterized by $\varepsilon\in\RR$.
Then, $W$ is differentiable at $\mu$ if and only if
\begin{equation}\label{eq:diff}
	D_\varphi W(\mu):=\tfrac{\d}{\d \varepsilon} W(\mu_{\varphi,\varepsilon})\big|_{\varepsilon=0}
\end{equation}
acts as a linear functional in its argument $\varphi$ and is continuous with respect to the
(trace) topology of $L^2(\RR^d\to\RR^d,\mu)$.
Behind this definition is the conception of $\scr P$ as an infinite-dimensional Riemannian-like structure with tangent bundle $(L^2(\RR^d\to\RR^d,\mu))_{\mu\in\scr P}$.
The inner product at $\mu\in\scr P$ is given by
\begin{equation*}
	\la \phi_1,\phi_2\ra_{L^2(\RR^d\to\RR^d,\mu)}:=\mu\big(\la\phi_1,\phi_2\ra\big):=\int_{\RR^d}\la\phi_1,\phi_2\ra\d\mu ,\qquad \phi_1,\phi_2\in L^2(\RR^d\to\RR^d,\mu),
\end{equation*}
where $\la\cdot,\cdot\ra$ denotes the Euclidean scalar product on $\RR^d$.
The intrinsic  gradient of $W$ at $\mu$ is then defined as the unique element $DW(\mu)\in L^2(\RR^d\to\RR^d,\mu)$ such that
\begin{equation}\label{eq:gradient}
	\la DW(\mu),\varphi\ra_{L^2(\RR^d\to\RR^d,\mu)}=	 D_\varphi W(\mu),\qquad \varphi\in C_b^1(\RR^d,\RR^d).
\end{equation} 
Regardless of a choice for a metric on $\scr P$, this can be taken as a definition of the gradient. 
We refer to \eqref{eq:diff} and \eqref{eq:gradient} as the intrinsic derivative, respectively the (intrinsic) gradient. 
A discussion about
the relation between the intrinsic and the extrinsic derivative can be found in \cite{RW20}.
For $p\in[1,\infty)$ and functions $W$ defined on the $p$-Wasserstein space 
\begin{equation*}
	\scr P_p:=\big\{\mu\in \scr P:\ \mu(|\cdot|^p) <\infty\big\}
\end{equation*}
the assignment $\varphi\mapsto D_\varphi W(\mu)$ as in \eqref{eq:diff}
is required to be continuous w.r.t.~the trace topology of $L^p(\RR^d\to\RR^d,\mu)$, rather than $L^2(\RR^d\to\RR^d,\mu)$, in order 
to define the gradient at $\mu$ (cf.~Definition \ref{D1} below).
This leads to a consistent notion of the gradient at $\mu$ (the class of differentiable functions depends on $p$, but not the value of $D_\varphi W(\mu)$ for $\varphi\in C_b^1(\RR^d,\RR^d)$ if defined). Naturally, \eqref{eq:diff} extends to all curves ${(\mu_{\varphi,\varepsilon})}_{\varepsilon\in\mathbb R}$ with $\varphi\in L^p(\RR^d\to\RR^d,\mu)$.

Sometimes, depending on the application, it is useful to consider an equivalent inner product at $\mu$, instead of the standard structure $L^2(\mathbb R^d\to\mathbb R^d,\mu)$,
leading to a perturbed value of the gradient $DW(\mu)$. For
a measurable weight function $\gamma:\RR^d\times\scr P\to[c^{-1},c]$, $c\in(0,\infty)$, we define 
\begin{equation*}
D^\gamma W(\mu):=\gamma(\cdot,\mu)DW(\mu).
\end{equation*}
This corresponds to defining the gradient of $W$ at $\mu$ based on the inner product of the weighted space $L^2(\RR^d\to\RR^d,\gamma(\cdot,\mu)^{-1}\mu)$.

By the findings in \cite{BR23}, under suitable conditions 
on the coefficients $\beta:\RR\to\RR$, $b:\RR\to\RR$ and $\Phi:\RR^d\to\RR$, the generalized porous media equation
	\begin{equation}\label{eq:gPME}
	\partial_t \rho=\Delta \beta(\rho)+\textnormal{div}\big((\nabla \Phi)b(\rho)\rho\big)\qquad\text{on } (0,\infty)\times\RR^d,
\end{equation}
has a mild solution $\rho(t,x)$, $t\in [0,\infty)$, $x\in\RR^d$, given by a nonlinear semigroup ${(S(t))}_{t\ge 0}$ of contractions in $L^1(\RR^d,\d x)$, i.e.
$\rho(t,\cdot)=S(t)\rho_0$ for an initial value $\rho(0,\cdot ):=\rho_0\in L^1(\RR^d,\d x)$. 
In \cite{BaR23} the uniqueness of this solution is shown in the largest class of solutions, namely the so-called distributional solutions. 
The positivity $\rho_0\ge 0$ and the mass $\int_{\RR^d}\rho_0(x)\d x$ of an initial value 
are preserved by $S(t)$. 
\eqref{eq:gPME} is an example from the class of nonlinear Fokker--Planck equations, which describe the evolution of the time marginal laws of 
solutions to distribution dependent stochastic differential equations, known as McKean--Vlasov SDEs. We refer to \cite{BaR24} (and also the references therein) for a comprehensive study of this topic.
By virtue of \cite{RR23} the curves  $\mu_t:=\rho(x,t)\d x\in\scr P$, $t\ge 0$, defined by the probability solutions to \eqref{eq:gPME} correspond
to the unique gradient flow on $\scr P$ satisfying
\begin{equation}\label{eq:I1}
	\frac{\d}{\d t}\mu_t=-D^\gamma W(\mu_t)
\end{equation}
w.r.t.~$\gamma(x,\rho\d x):=b(\rho(x))$ and an explicit energy functional $W$. The latter has a domain within the set of absolutely continuous measures on which
\begin{equation*}
	W(\rho\d x):=\int_{\RR^d} F(x,\rho(x))\d x
\end{equation*}
with
\begin{equation}\label{eq:FRR1}
	F(x,s):=s\Phi(x)+\int_0^s\int_1^t\frac{\beta'(r)}{rb(r)}\d rd t,\quad x\in\RR^d,\,s\in(0,\infty),
\end{equation}
is well-defined.
The energy functional $W$ coincides with the Lyapunov function determined in \cite{BR23} (see, in particular, \cite[Thm.~4.1]{BR23} for the corresponding energy dissipation inequality 
\cite[(4.10)]{BR23}).
For $\Phi=0$, $b=\beta'=1$, \eqref{eq:gPME} reduces to the heat equation with $W(\rho\d x):=\int_{\RR^d}(\ln(\rho)-1)\rho\d x$ being the entropy functional.
If a suitable space of test functions $u$ is given,
\eqref{eq:I1} can be reformulated into
\begin{equation*}
	\frac{\d}{\d t}u(\mu_t)=-\int_{\RR^d}\la\gamma_{\mu_t} DW(\mu_t),Du(\mu_t)\ra\d\mu_t,
\end{equation*}
where we set $\gamma_\mu:=\gamma(\cdot,\mu)$.

\subsection{From deterministic to stochastic gradient flows}
We are interested in constructing stochastic gradient flows for energy functionals of the type 
$W_F(\mu):= \int_{\RR^d}F(x,\rho_\mu(x))\d x$, $\mu=\rho_\mu\d x$, with  $F:\R^d\times [0,\infty)\to\RR$ for example as in \eqref{eq:FRR1}.
Adding a stochastic process $({R_t})_{t\ge 0}$ to its right-hand side transforms \eqref{eq:I1} into a stochastic differential equation
\begin{equation}\label{eq:SGF0}\d \mu_t=  -D^\gamma W_F(\mu_t)\d t+ \d R_t,\qquad t\ge 0.\end{equation} 
In the main part of this article we work with an arbitrary choice $p\in[1,2]$ and
let ${(R_t)}_{t\ge 0}$ be a symmetric Markov process with state space $\scr P_p$ and invariant measure $\Lam$, which is induced by a restricted Ornstein--Uhlenbeck process
as explained below. For some particular results we focus on the most relevant case $p=2$ for simplicity.
Denoting the generator of ${(R_t)}_{t\ge 0}$ by $A$ we reformulate \eqref{eq:SGF0} as a martingale problem for the generator of ${(\mu_t)}_{t\ge 0}$.
In this article, a solution to \eqref{eq:SGF0} is constructed in the sense of a right process
$(\Omega, \scr F,(\mu_t)_{t\geq 0},(\P^\mu)_{\mu\in\scr P_p})$ on $\scr P_p$ for which, 
\begin{equation}\label{eq:M}
	u(\mu_t)-u(\mu_0)-\int_0^t Au(\mu_s)\d s-\mu_s\big(\gamma_{\mu_s}\la DW_F(\mu_s),Du(\mu_s)\ra\big)\d s,\qquad t\geq 0,
\end{equation}
is a martingale for  $u\in\D(A)$ and starting points $\mu$ up to a set of zero capacity.
There is no Brownian motion on the set of probability measures. Hence, we do not have a canonical candidate for ${(R_t)}_{t\ge 0}$.
A substitute must be chosen with care, because the energy functional $W_F$
does not have an intrinsic derivative $DW_F$ defined at all points of $\scr P_p$, but only within a certain subset of absolute continuous measures.
It is a priori not clear how ${(R_t)}_{t\ge 0}$ can be defined in a way that martingale solutions to \eqref{eq:SGF0} in the sense of \eqref{eq:M} exist.

We utilize the method in \cite{RW22, RWW24} by which
a Markov process ${(\phi_t)}_{t\ge 0}$ on $L^p(\RR^d\to\RR^d,\lam)$ for fixed $\lam\in\scr P_p$ induces a Markov process on $\scr P_p$.
Here, we take ${(\phi_t)}_{t\ge 0}$ as a restricted Ornstein--Uhlenbeck process on
$L^p(\RR^d\to\RR^d,\lam)$ in such a way that \eqref{eq:M} becomes solvable with $A$ being the generator of the induced process ${(R_t)}_{t\ge 0}$ on $\scr P_p$.
The test function $u$ must be an element in the $L^2$-domain $\D(A)$ of the generator, which makes $(A,\D(A))$ a non-positive self-adjoint operator in $L^2(\scr P_p,\Lam)$.
The Dirichlet form of ${(R_t)}_{t\ge 0}$ is given by
\begin{equation}\label{eq:genA}
	\la -Au,v\ra_{L^2(\Lam)}=\int_{\scr P_p}\mu\big(\gamma_\mu\la Du(\mu),Dv(\mu)\ra\big)\Lam(\d\mu),\qquad u,v\in\D(A).
\end{equation}

Let us now briefly point out how to find a solution to \eqref{eq:M} once a suitable ${(R_t)}_{t\ge 0}$ has been fixed. 
We consider a perturbation
\beq\label{eq:LW}\LL_F(\d\mu):= \ff 1 {Z_F} \e^{-W_F(\mu)}\LL(\d\mu)\end{equation}
of the invariant measure $\Lam$, where for $\mu=\rho_\mu\d x$ we set
\begin{equation}\label{eq:W}W_F(\mu):=\beg{cases} \int_{\RR^d}F(x,\rho_\mu(x))\d x\ &\text{if}\ F(\cdot,\rr_\mu)\in
L^1(\R^d,\d x),\\
\infty\ &\text{otherwise,}\end{cases}  \end{equation}
and \begin{equation*}Z_F:= \int_{\scr P_p} \e^{-W_F}\d\LL.\end{equation*}
Our choice of ${(R_t)}_{t\ge 0}$ respectively $\Lambda$ (see below) ensures that $0<Z_F<\infty$.
The calculus for  the square-field operator
\begin{equation*}
	\Gamma(u,v)(\mu):=\mu\big(\gamma_\mu\la Du(\mu),Dv(\mu)\ra\big),
\end{equation*}
together with \eqref{eq:genA} leads to
\begin{equation}\label{eq:genB}
	\big\la -Au+\mu\big(\gamma_\mu\la DW_F(\mu),Du(\mu)\ra\big),v\big\ra_{L^2(\Lam_F)}=\int_{\scr P_p}\Gamma(u,v)\Lam_F(\d\mu)
\end{equation}
under suitable assumptions on $F,u,v$ (cf.~Lemma \ref{lem:generator} and the proof of Corollary  \ref{cor:SFG} below). So, formally a $\Lam_F$-symmetric process ${(\mu_t)}_{t\ge 0}$ solves \eqref{eq:M} if its Dirichlet form $\EE^F$ reads
\beq\label{eq:Dform}\EE^F(u,v):= \int_{\scr P_p }\GG(u,v) \d\LL_F.\end{equation}
This method is well-known in the theory of Dirichlet forms and has been applied in many different settings, e.g.~in \cite{AHS,F81, AR91, RZ92, E96, BG14}.
The process ${(\mu_t)}_{t\ge 0}$ can be obtained by: (1) proving closability of the bilinear form $\EE^F$ in \eqref{eq:Dform} on a core of differentiable functions and (2) proving quasi-regularity 
(see \cite[Chap.~IV]{MR92}) of the minimal closed extension of $\EE^F$ in $L^2(\scr P_p,\Lam_F)$.

To realize this plan, we choose ${(R_t)}_{t\ge 0}$ and $\Lam$ as follows (assuming $\gamma(\cdot,\cdot)=1$ in this paragraph for simplicity).
First, a non-degenerate Gaussian measure $G$ on $$\{\phi\in C^1(\RR^d,\RR^d):\|\nn \phi\|_\infty<\infty\}$$ is fixed,
where 
\begin{equation*}
\|\nn \phi\|_\infty=\sup_{x\ne y}\ff{|\phi(x)-\phi(y)|}{|x-y|},\qquad \phi:\RR^d\to\RR^d.
\end{equation*}
Then, $G$ assigns a strictly positive value to the set
\begin{align}\label{eq:DDp}
\DDp:=\Big\{\phi\in C^1(\RR^d,\RR^d): \, \phi\ \text{is\ invertible, } \phi^{-1}\in C^1(\RR^d,\RR^d),&\\ 
\text{and }\|\nn \phi\|_\infty+\|\nn (\phi^{-1})\|_\infty<\infty & \Big\},\nonumber
\end{align} 
see Remark \ref{rem:DDnew}.
Let $\ll\in \scr P_p$ be absolutely continuous
(by the observation in Remark \ref{rem:Gauss} it doesn't matter which element $\lam$ with that property we fix here). 
A Gaussian-based measure $\LL$ is defined as the push-forward of the probability measure
$\ff{\eins_{\DDp}(\phi)G(\d\phi)}{G(\DDp)}$  
under the map \begin{equation*}\Psi_\lam:\DDp\ni \phi\mapsto\lam\circ\phi^{-1}\in\scr P_p.\end{equation*}
Since $\DDp\subset L^p(\RR^d\to\RR^d)$ for $p\in[1,2]$ as well as
\begin{equation*}\big(L^p(\RR^d\to\RR^d,\lam)\big)^*\subseteq \big(L^2(\RR^d\to\RR^d,\lam)\big)^*=L^2(\RR^d\to\RR^d,\lam)\subseteq L^p(\RR^d\to\RR^d,\lam),\end{equation*}
the classical Gaussian gradient-type bilinear form on $\DDp$ is defined as
\begin{equation*}
	\tilde\EE(f,g)=\int_{\DDp}\la\nabla f(\phi),\nabla g(\phi)\big\ra_{L^2(\RR^d\to\RR^d\lam)}\ff{\eins_{\DDp}(\phi)G(\d\phi)}{G(\DDp)}
\end{equation*}
for $f,g$ in a suitable pre-domain, see \cite[Sect.~II.3]{MR92}. It has a minimal closed extension $(\tilde\EE,\D(\tilde\EE))$ in $L^2(\DDp,\ff{\eins_{\DDp}G}{G(\DDp)})$
and there is a restricted Ornstein--Uhlenbeck process ${(\phi_t)}_{t\ge 0}$  associated with the Dirichlet form $(\tilde\EE,\D(\tilde\EE))$.
Following the concept of \cite{RW22, RWW24} there exists a diffusion process ${(R_t)}_{t\ge 0}$ on $\scr P_p$ whose associated Dirichlet form $(\EE,\D(\EE))$ in $L^2(\scr P_p,\Lam)$
is given by
\begin{equation*}
	\EE(u,v)=\int_{\scr P_p}\la Du(\mu), Dv(\mu)\big\ra_{ L^2(\RR^d\to\RR^d,\mu)}\Lam(\d\mu)
\end{equation*}
for $u,v$ in a suitable pre-domain.
${(R_t)}_{t\ge 0}$ is referred to as the induced process of ${(\phi_t)}_{t\ge 0}$ in this text, because $\EE$ coincides with the image Dirichlet form of $\tilde\EE$ under $\Psi_\lam$ on its domain $\D(\EE)$.

The set-up described in the paragraph above is sufficient to prove the existence of a diffusion on $\scr P_p$ whose Dirichlet form is as in \eqref{eq:Dform}. Formally, the process solves \eqref{eq:SGF0}. 
To ensure the integrability of the drift term we make an amendment regarding the definition of $\DDp$.
By ${(R_t^{(n)})}_{t\ge 0}$ we denote the Markov process induced by a restricted Ornstein--Uhlenbeck process in the same way as above, but regarding a Gaussian measure 
restricted to
\begin{equation*}
	\D^{(n)}:=\Big\{\phi\in \DDp\cap C^2(\RR^d,\RR^d): 
	|\phi(0)|+\|\na\phi\|_\infty+\|\na^2\phi\|_\infty+\|\nn(\phi^{-1})\|_\infty<n\Big\}
\end{equation*}
instead of $\DDp$, where $n\in\NN_{\geq 3}$ is arbitrary but fixed. The existence of a Gaussian measure which assigns a positive value to the event $\D^{(n)}$ is discussed in Section \ref{sec:Loc}. We solve \eqref{eq:M} for $A=A^{(n)}$, the generator of ${(R_t^{(n)})}_{t\ge 0}$.

\subsection{Outline and main results}
The functional $W_F$ in \eqref{eq:W} with $F$ as in \eqref{eq:FRR1} is of major importance, 
since $W_F$ is a Lyapunov function for the generalized porous media equation (see \cite{BR23}).
In this context, our main results can be summarized as follows.
\begin{itemize}
\item We first define a class of Gaussian-based measures $\Lam$ on $\scr P_p$, $p\in[1,2]$, (see Section \ref{ssec:Gauss})
and then construct a stochastic quantization for the measure $\Lam_F$, i.e.~a symmetric diffusion process ${(\mu_t)}_{t\ge 0}$ on $\scr P_p$ with invariant measure $\Lam_F$ as given in \eqref{eq:LW},
by showing:
\begin{itemize}
	\item The bilinear form $\EE^F$ in \eqref{eq:Dform} with the domain of bounded, continuously differentiable functions on $\scr P_p$ (see Definition \ref{D1}) is well-defined and closable in 
	$L^2(\scr P_p,\Lam_F)$.
	\item Its minimal closed extension $(\EE^F,\D(\EE^F))$ is a 
	local, conservative and quasi-regular Dirichlet form in $L^2(\scr P_p,\Lam_F)$.
	Thanks to the one-to-one correspondence 
	between the family of local, quasi-regular Dirichlet forms and the family of diffusion processes on a topological Lusin space
	 (see \cite[Chap.'s~IV \& V]{MR92}), we obtain a ($\Lam_F$-symmetric) conservative diffusion process 
	$\mathbf M=(\Omega, \scr F,(\mu_t)_{t\geq 0},(\P^\phi)_{\phi\in \scr P_p})$ 
	 on $\scr P_p$, properly associated with $(\EE^F,\D(\EE^F))$.
\end{itemize}
The two statements above are proven in Theorem \ref{thm:main} and Proposition \ref{prp:basicP}. 
The case with $F$ as in \eqref{eq:FRR1} is treated in Example \ref{exa:F2}, where we choose $p=2$.

\item Under suitable conditions the Gaussian-based measure $\Lam$ has the property that the Radon--Nikodym derivative  $\rr_\mu=\ff {\d\mu}{\d x}$ exists and is Lipschitz continuous for $\LL\text{-a.e.~}\mu$.  The  gradient
of the energy functional for the generalized porous media equation \eqref{eq:gPME} is defined for $\LL\text{-a.e.~}\mu$ 
and computed as
\begin{equation*}
	D W_F(\mu)= \na \Phi+\frac{\beta'(\rho_\mu)\na\rho_\mu}{b(\rho_\mu)\rho_\mu}
\end{equation*} in the sense of a local weak intrinsic gradient  (see Definition \ref{def:LIWD}, Theorem \ref{thm:localDom} and the subsequent discussion). 
\item For each $n\in\NN_{\geq 3}$ there is a Gaussian-based measure  $\LL^{(n)}$ (see Section \ref{sec:Loc}, in particular Corollary \ref{cor:locSGH}) such that:
\begin{itemize}
\item  The resulting process ${(\mu^{(n)}_t)}_{t\ge 0}$ formally solves 
\begin{equation}\label{eq:SGFn}
\d \mu^{(n)}_t=  -DW_F(\mu_t^{(n)})\d t+ \d R^{(n)}_t,\qquad t\ge  0,
\end{equation}
where ${(R_t^{(n)})}_{t\ge 0}$ is induced by a restricted Ornstein--Uhlenbeck process on $\D^{(n)}$ in the sense explained above.
\item  $E_n:=\supp[\LL^{(n)}]$ are monotone increasing in $n$ and $\bigcup_{n\geq 3}E_n\subset\scr P_2$ is dense.
\end{itemize}
\end{itemize}

The remainder of this article is organized as follows.
Section \ref{ssec:general} contains a summary of preliminaries on strongly local Dirichlet forms and their connection with diffusion processes.
In Section \ref{ssec:indT} we state a quasi-regularity result related to the intrinsic derivative for Dirichlet forms on the Wasserstein space.
The class of suitable reference measures $\Lam$ on $\scr P_p$, including Gaussian-based measures, are introduced in Section \ref{ssec:Gauss}.
This framework is applied in Section \ref{sec:EF} to
prove the existence of a diffusion with the associated Dirichlet form  in \eqref{eq:Dform} for the energy functional given  in \eqref{eq:W} (Theorem \ref{thm:main}).
Localized versions of the Dirichlet form in \eqref{eq:Dform} are introduced in Section \ref{sec:Loc}. As a consequence, we obtain the local weak gradient of the functional \eqref{eq:W} and show \eqref{eq:SGFn} (Theorem \ref{thm:localDom} and Corollary \ref{cor:locSGH}).

\section{Diffusion processes  on Wasserstein spaces}\label{sec:diffusion}

To develop stochastic analysis on a state space of measures it is crucial to find reasonable measure-valued  processes which may play the role 
of Brownian motion to drive  SDEs.
In the literature, several types of  measure-valued diffusion processes are  constructed 
by means of (quasi-)regular, local Dirichlet forms in $L^2$-spaces which have reference probabilities  supported on the set of singular distributions, see  \cite{D22, RW20, RS07, St24} and the references therein. Moreover, measure-valued diffusion processes over $\R$ with strong Feller properties are analyzed in \cite{Del, D20, VWW}.

In this section  we introduce a new class of diffusions on the Wasserstein space.
The most important examples of this class are processes with a Gaussian-based invariant measure.

\subsection{Preliminaries}\label{ssec:general}

Let $E$ be a metrizable Lusin topological space. We denote its Borel $\sigma$-algebra by $\B(E)$ and  fix a probability measure $\Lam$  on $(E,\B(E))$.
Given a measurable function $f:E \to\RR$, the corresponding $\Lam$-class of measurable functions is again denoted by $f$. The positive part of $f$ is denoted by $f^+$.
We set $f\land g:=\min\{f,g\}$ for  $f,g:E \to\RR$.
All vector spaces in this text are assumed to be real.
Let  $\D(\EE) $ be a dense subspace of   $L^2(E,\Lam)$ and $$\EE:\D(\EE)\times\D(\EE)\to\RR$$ be a non-negative definite, symmetric bilinear map.   
$(\EE,\D(\EE))$ is called closed if $\D(\EE)$ is complete under the $\EE_1^{1/2}$-norm which is induced by the inner product
$$\EE_1(u,v):=\EE(u,v)+{\la u,v\ra}_{L^2(E,\Lam)}.$$
It  is    called a symmetric Dirichlet form if it is closed and 
$$u^+\land \eins \in\D(\EE)\qquad \text{with}\qquad  \EE(u^+\land \eins,u^+\land \eins)\leq \EE(u,u)$$
for $u\in\D(\EE)$.

Since $E$ has the strong Lindel\"of property, the support of a positive measure on $(E,\B(E))$ is well-defined and denoted by $\supp[\cdot]$. For a measurable function $f:E\to\RR$ we set
$\supp[f]:=\supp[|f|\Lam]$.
A symmetric Dirichlet form $(\EE,\D(\EE))$ is said to possess the local property if
\begin{equation*}
	\EE(u,v)=0\qquad\text{for }u,v\in\D(\EE):\supp[u]\cap\supp[v]=\varnothing.
\end{equation*}
$(\EE,\D(\EE))$  is said to possess the strong local property if
\begin{equation*}
	\EE(u,v)=0\qquad\text{for }u,v\in\D(\EE):u\textnormal{ is constant }\Lam\text{-a.e.~on }\supp[v].
\end{equation*}
We call $(\EE,\D(\EE))$ conservative if $\eins\in \D(\EE)$ and $\EE(\eins,\eins)=0$. In the conservative case the strong local and the local property are equivalent.

Let $(\EE,\D(\EE))$ be a Dirichlet form on $E$. If there exists a bilinear map
$$\Gamma:\D(\EE)\times\D(\EE)\to L^1(E,\Lam)$$ such that
\begin{equation}\label{eq:squareF}
	\EE(uw,v)+\EE(vw,u)-\EE(w,uv)=2\int_Ew\Gamma(u,v)\d\Lam,\quad u,v,w\in\D(\EE)\cap L^\infty(E,\Lam),
\end{equation}
then $(\Gamma,\D(\EE))$ is called the square-field operator of $(\EE,\D(\EE))$. In the conservative case \eqref{eq:squareF} implies
\begin{equation}\label{eg:Gammascale}
 \EE(u,v)=\int_E \GG(u,v)\d\LL,\qquad u,v\in \D(\EE).\end{equation}
Let it be remarked that in some textbooks the definition of $\Gamma$ includes a scaling factor $\frac 1 2 $ on the right-hand side of \eqref{eg:Gammascale}.
We note that \eqref{eq:squareF} is true if it is satisfied for all $u,v,w\in \scr A$
for some subalgebra $\scr A$ of $\D(\EE)\cap L^\infty(E,\Lam)$ which is dense w.r.t.~$\EE_1^{1/2}$-norm.
Moreover, \eqref{eq:squareF} implies continuity of $\Gamma$ (w.r.t.~$\EE_1^{1/2}\times\EE_1^{1/2}$).
We set
\begin{equation*}
\D_{b,\lip}(\EE):=\big\{u\in\D(\EE)\cap L^\infty(E,\Lam):\Gamma(u,u)\in L^\infty(E,\Lam)\big\}.
\end{equation*}
According to   the results in \cite[Sect.~I.5]{BH91} a strongly local symmetric Dirichlet form with square-field operator
$\GG$ satisfies
\begin{equation*}
vw\in\D(\EE)\qquad\text{and}\qquad \Gamma(u,vw)=v\Gamma(u,w)+w\Gamma(u,v)\in L^1(E,\Lam)
\end{equation*}
for $u,v\in\D(\EE)$ and $w\in \D_{b,\lip}(\EE)$.

The generator of $(\EE,\D(\EE))$ is the unique non-positive definite, self-adjoint
operator \sloppy$(L,\D(L))$ in $L^2(E,\Lam)$ such that $\D(\EE)=\D(\sqrt{-L})$ and $\EE(u,v)={\la\sqrt{-L}u,{\sqrt{-L}v}\ra}_{L^2(E,\Lam)}$. It is a Dirichlet operator,
i.e.~${\la Lu,(u-1)^+\ra }_{L^2(E,\Lam)}\leq 0$ for $u\in\D(L)$, and the infinitesimal generator of a strongly continuous semigroup ${(T_t)}_{t\geq 0}$ 
of $\Lam$-symmetric, sub-Markovian operators in $L^2(E,\Lam)$ with
\begin{equation*}
0\leq T_tu\leq 1\quad (\Lam \text{-a.e.})\quad\text{whenever}\quad 0\leq u\leq 1\quad(\Lam\text{-a.e.}).
\end{equation*}
The family of Dirichlet forms in $L^2(E,\Lam)$ and the family strongly continuous semigroups ${(T_t)}_{t\geq 0}$
which consist of $\Lam$-symmetric, sub-Markovian operators stand in one-to-one correspondence. $(\EE,\D(\EE))$ conservative if and only if  $T_t\eins=\eins$.

The operators $(T_t)_{t\ge 0} $ are contractive w.r.t.~$\|\cdot\|_{L^p(E,\Lam)}$ for $p\in [1,\infty]$ and hence uniquely determine a
strongly continuous contraction semigroup in $L^p(E,\Lam)$ for $p\in[1,\infty)$ (the continuity of this semigroup can be shown if $p<\infty$). Any of these, on their domain $L^p(E,\Lam)$, coincide with the extension of $(T_t)_{t\ge 0} $ to $L^1(E,\Lam)$, because $L^p(E,\Lam)\subset  L^1(E,\Lam)$.
Therefore, we do not distinguish  in notation and write  ${(T_t)}_{t\geq 0}$ indifferent of the index $p$.
Its infinitesimal generator in $L^1(E,\Lam)$ coincides with the smallest closed extension of $(L,\D(L))$ in $L^1(E,\Lam)$, see \cite[Prop.~2.4.2]{BH91}, and is denoted by
$(L,\D_{\textnormal{op,}L^1}(L))$ in the following paragraph.

Dirichlet forms are used to analyze symmetric Markov processes. The behaviour 
under a modification of the invariant measure through the multiplication with a density is, for example, discussed in \cite{AR91, T95, E96, F97}.
We include an elementary and easy-to-prove result on the behaviour of the generator under such a transform here, as it is relevant for Sections \ref{sec:EF} \& \ref{sec:Loc}.
Let $\Lam_\circ$ be another probability measure on $(E,\B(E))$ and $(\EE^\circ,\D(\EE^\circ))$
be a Dirichlet form in $L^2(E,\Lam_\circ)$.
Furthermore, we assume:
\begin{itemize}
	\item[(a)] $\Lam$ is absolutely continuous w.r.t.~$\Lam_\circ$ with Radon--Nikodym density satisfying
	$\varrho:=\frac{\d\Lam}{\d\Lam_\circ}>0$ $\Lam_\circ$ a.e.~and $\varrho\in\D(\EE^\circ)$.
	\item[(b)] $(\EE^\circ,\D(\EE^\circ))$ is a conservative, strongly local Dirichlet form  with generator $(A,\D(A))$ and  square-field operator $\Gamma$.
	 $(\EE,\D(\EE))$ is a Dirichlet form in $L^2(E,\Lam)$ such that there exists a subspace $\scr L\subseteq\D_{b,\lip}(\EE^\circ)\cap\D(\EE)$, densely included in $\D(\EE)$ w.r.t.~$ \EE_1^{1/2}$
	 and
		$\EE(u,v)=\int_E\Gamma(u,v)\d\Lam$, $u,v\in\scr L.$
\end{itemize}

\noindent As above, the generator of the sub-Markovian semigroup ${( T_t)}_{t\geq 0}$ in $L^1(E,\Lam)$ corresponding to $( \EE,\D( \EE))$ is denoted by $( L,\D_{\textnormal{op,}L^1}(L))$.

\begin{lem}\label{lem:generator} Assume \textnormal{(a)} and \textnormal{(b)}. 
	We have
	\begin{equation*}
		\D(\EE)\cap\D(A)\subseteq \D_{\textnormal{op,}L^1}(L)\qquad\text{and}\qquad Lu=A u+\varrho^{-1}\Gamma(u,\varrho),\quad u\in \D(\EE)\cap\D(A).
	\end{equation*}
\end{lem}
\begin{proof} We fix $u\in\D(\EE)\cap\D(A)$. For any $v\in\scr L$ it holds
	\begin{equation*}
		\int_E\Gamma(u,v)\d \Lam=\int_E\Gamma(u,v\varrho)-v\Gamma(u,\varrho)\d\Lam_\circ
		=\int_E\big(-Au-\varrho^{-1}\Gamma(u,\varrho)\big)v\d\Lam.
	\end{equation*}
	Thus, the equality $\EE(u,v)=\int_Ef_uv\d\Lam$ with $f_u:=-Au-\varrho^{-1}\Gamma(u,\varrho)$ holds for all $v\in\D( \EE)\cap L^\infty(E,\Lam)$
	by density of $\scr L$.
	In particular,
	\begin{equation*}
		\int_E( L T_tu)v\d\Lam=
		\int_E Lu T_tv\d\Lam=-\EE (u, T_tv)=
		\int_E(-f_u) T_tv\d\Lam=
		\int_E(- T_tf_u)v\d\Lam.
	\end{equation*}
	for $v\in\D(\EE)\cap L^\infty(E,\Lam)$, $t>0$,
	and so, $L T_tu=- T_tf_u$ in $\Lam$-a.e. The claim follows with strong continuity of the semigroup,
	\begin{equation*}
		T_tu\overset{t\to 0}{\longrightarrow} u\quad\text{and}\quad  T_tf_u\overset{t\to 0}{\longrightarrow}f_u\quad \text{in }L^1(E,\Lam),
	\end{equation*}
	and the fact that $( L,\D_{\textnormal{op,}L^1}(L))$ is a closed operator.
\end{proof}

We recall the relevant basic notions from \cite{MR92}, \cite{FOT11}.
We are only interested in the symmetric case. 
The set of (bounded/non-negative) measurable functions $(E,\B(E))\to \RR$ are denoted by $\B(E)$ (respectively $\B_b(E)$/$\B_+(E)$).
Let 
\begin{equation*}
\mathbf M=(\Omega, \scr F,(\mu_t)_{t\geq 0},(\P^\phi)_{\phi\in E_\Delta})
\end{equation*} be a right process with state space $E$, life time $\zeta$ and shift operator 
$\theta_t:\Omega\to\Omega$, $t\ge 0$, as defined in \cite[Def's.~1.5 \& 1.8]{MR92}.
The expectation  w.r.t.~$\P^\phi$ is denoted by $\E^\phi$.
Then,
\begin{equation*}
	P_tf(\phi):=\E^\phi[f(\mu_t)],\quad \phi\in E,\,f\in\B_+(E),\,t\geq 0,
\end{equation*}
defines a semigroup $(P_t)_{t\ge 0}$ of sub-Markovian kernel operators.
If ${(P_t)}_{t\geq 0}$ is $\Lam$-symmetric (in particular it respects $\Lam$-classes), then its action on $\B_b(E)$
uniquely determines a strongly continuous contraction semigroup ${(T_t)}_{t\geq 0}$ in $L^2(E,\Lam)$.
The process $\mathbf M$ is called associated with the Dirichlet form $(\EE,\D(\EE))$ which corresponds to ${(T_t)}_{t\geq 0}$.
By virtue of \cite[Thm.~IV.6.7]{MR92}, given a Dirichlet form $(\EE,\D(\EE))$ in $L^2(E,\Lam)$, the existence of a right process $\mathbf M$ associated with $\EE$ is equivalent to
the quasi-regularity of the form (defined as in \cite[Thm.~IV.3.1]{MR92}). 

We call an increasing sequence ${(F_k)}_{k\in\NN}$ of closed sets an $\EE$-nest if
\begin{equation*}
	\bigcup_{k\in\NN}\Big\{u\in\D(\EE):u(\phi)=0\text{ for }\Lam\text{-a.e.~}\phi\in E\setminus F_k\Big\}
\end{equation*} 
is dense in $(\D(\EE),\EE_1^{1/2})$. A subset $N\subseteq \bigcap_{k\in\NN}(E\setminus F_k)$ is referred to as $\EE$-exceptional. 
A statement depending on a reference point $\phi\in E$ is said to hold $\EE$-quasi-everywhere ($\EE$-q.e.) if valid for all $\phi\in E\setminus N$ for some $\EE$-exceptional set $N\subset E$.
The term $\EE$-quasi-continuous applies to a function $f:E\to\RR$ which restricts to a continuous function, $f|_{F_k}\in C(F_k)$, on $F_K$ for all $k\in\NN$.
If the right process $\mathbf M$ is associated with $(\EE,\D(\EE))$, then
its transition function $P_tf$  is $\EE$-quasi-continuous for all $f\in\B_b(E)$, $t>0$.
A quasi-regular Dirichlet form uniquely determines a $\Lam$-equivalence class of associated right processes on $E$, see \cite[Sect.~IV.6]{MR92}.
Within such a class, $\mathbf M$ can w.l.o.g.~be taken to be $\Lam$-tight and special standard (\cite[Def.~IV.1.13]{MR92}).
Then, $(\EE,\D(\EE))$ has the local property if and only if
\begin{equation}\label{eq:paths}
	[0,\zeta)\ni t\mapsto \mu_t\in E\text{ is continuous }\P^\phi\text{-almost surely}
\end{equation}
for $\EE$-q.e.~$\phi\in E$. We call a $\Lam$-tight and special standard process $\mathbf M=(\Omega, \scr F,(\mu_t)_{t\geq 0},(\P^\phi)_{\phi\in E})$
a non-terminating diffusion if \eqref{eq:paths} holds with $\zeta=\infty$ for all $\phi\in E$.
For a quasi-regular, local and conservative Dirichlet form  $(\EE,\D(\EE))$ there exists a non-terminating diffusion $\mathbf M$ which is  associated.

In the following we assume that  $(\EE,\D(\EE))$ is quasi-regular, local, conservative and admits a square-field operator $(\Gamma,\D(\EE))$.
$\{\scr F_t\}_{t\geq 0}$ denotes the minimum completed admissible filtration of $\mathbf M$.
We recall the definition of an additive functional ${(A_t)}_{t\geq 0}$  of $\mathbf M$. The term refers to a numerical function $(\omega,t)\mapsto A_t(\omega)$, $\omega\in\Omega$, $t\in[0,\infty)$, such that:
\begin{itemize}
	\item $A_t$ is $\scr F_t$-measurable. 
	\item $\exists B\in\scr F$: $\theta_t(B)\subseteq B$ $\forall$ $t>0$ and $\P^\phi(B)=1$ for $\EE$-q.e.~$\phi\in E$.
	\item There exists $B\in\scr F$ (defining set) and an $\EE$-exceptional set $N\subseteq E$ such that $\theta_t(B)\subseteq B$ for $t>0$,  $\P^\phi(B)=1$ for $\phi\in E\setminus N$,
	and for  $\omega\in B$:
	\begin{itemize}
		\item $Y_0(\omega)=0$,
		\item $Y_t(\omega)\in\R$, $Y_{t+s}(\omega)=Y_t(\omega)+Y_s(\theta_t\omega)$ for $s,t\geq 0$,
		\item ${(Y_t(\omega))}_{t\geq 0}$ is c\`adl\`ag.
	\end{itemize}
\end{itemize}
The set $B$ as above is called defining set for ${(A_t)}_{t\geq 0}$. If additionally $A_t(\omega)\geq 0$ for $t\geq 0$ and $\omega\in B$, then ${(A_t)}_{t\geq 0}$ is called positive.
${(A_t)}_{t\geq 0}$ is called continuous if ${(A_t(\omega))}_{t\geq 0}$ is continuous for $\omega\in B$.
Moreover, ${(A_t)}_{t\geq 0}$ is said to be of finite (resp.~zero) energy if the limit
\begin{equation*}
	\lim_{t\to\infty}\frac{1}{2t}\int_\Omega A_t^2(\omega)\P_\Lam(\d\omega)\in[0,\infty)
\end{equation*}
exists (resp.~vanishes), where $\P_\Lam$ is the equilibrium measure $\P_\Lam(\d\omega):=\int_{E} \P^\phi(\d\omega)\Lam(\d\phi)$ on $\Omega$.
Two additive functionals $A^{(1)}={(A^{(1)}_t)}_{t\geq 0}$ and ${A^{(2)}}={(A^{(2)}_t)}_{t\geq 0}$ are called equivalent if they admit a common defining set $B$ such that
$A^{(1)}_t(\omega)=A^{(2)}_t(\omega)$ for $t\geq 0$, $\omega\in B$.  We 
write $A^{(1)}=A^{(2)}$ in that case and   identify an additive functional with its equivalence class.

A positive Radon measure $m$ on $(E,\B(E))$ is called smooth if any $\EE$-exceptional set is a $m$-nullset and there is an $\EE$-nest $\{F_k\}_k$ of compact sets such that
$m(F_k)<\infty$ for all $k$. The Revuz correspondence (see \cite[Chap.~VI, Thm.~2.4]{MR92}, \cite[Thm's.~5.1.3 \& 5.1.4]{FOT11}, \cite[Thm.~4.1.1]{CF}) describes a one-to-one assignment between the family of smooth measures $m$ and the family of positive continuous additive functionals.

\begin{rem}\label{rem:Revuz}
	If $u\in L^1(E,\Lam)$, then $u\Lambda$  is a smooth measure.
	The unique continuous additive functional determined by $u\Lambda$ is (up to equivalence) given by
	\begin{equation*}
		\Omega\times [0,\infty)\ni (\omega,t)	\mapsto \int_0^tu(\mu_s(\omega))\d s.
	\end{equation*}
	At first, this characterization can be shown if $u\in L^2(E,\Lambda)$, see \cite[Sect.~5.1]{FOT11}, and then generalizes (in particular) to $u\in L^1(E,\Lambda)$ by approximation, see \cite[Sect.'s~6 \& 7]{Uem} for details.
\end{rem}

Let $u\in\D(\EE)$ and $\tilde u$ denote a $\EE$-quasi-continuous representative (such $\tilde u$ always exists by \cite[Chap.~IV, Prop.~3.3]{MR92}).
Then, ${(\tilde u(\mu_t)- \tilde u(\mu_0))}_{t\geq 0}$ defines  an additive functional which
up to equivalence does not depend on the choice of a quasi-continuous modification of $u$  (see \cite[Sect.~5.2 (I)]{FOT11}).
 The Fukushima decomposition (\cite[Thm.~2.5]{MR92}, \cite[Thm.~5.2.2]{FOT11}) states that
\begin{equation}\label{eq:FukuD}
	\tilde u(\mu_t)-\tilde u(\mu_0)=M_t^{[u]}+N_t^{[u]},\qquad t\geq 0,
\end{equation}
where:
\begin{itemize}
	\item ${(N_t^{[u]})}_{t\geq 0}$ is a continuous additive functional of zero energy and
	 $N_t^{[u]}\in L^1(\Omega,\P^\phi)$ for all $t\geq 0$ holds for $\EE$-q.e.~$\phi\in E$.
	\item ${(M_t^{[u]})}_{t\geq 0}$ is an additive functional of finite energy, Moreover, $\int_\Omega M_t^{[u]}(\omega) \P^\phi(\d\omega)=0$ and $M_t^{[u]}\in L^2(\Omega,\P^\phi)$ 
	for all $t\geq 0$ hold for $\EE$-q.e.~$\phi\in E$.
\end{itemize}
Moreover, ${(M_t^{[u]})}_{t\geq 0}$, ${(N_t^{[u]})}_{t\geq 0}$ in \eqref{eq:FukuD} with the properties above are unique up to equivalence.

\begin{rem}\label{rem:MS}
	(i) If $u\in \D_{\textnormal{op},L^1}(L)\cap\D(\EE)$, where $L$ is the generator of $\D(\EE)$ and $\D_{\textnormal{op},L^1}(L)$ is the domain of a closed operator extension of $L$ in $L^1(E,\Lam)$, then 
\begin{equation}\label{eq:zeroE}
	N_t^{[u]}=\int_0^tLu(\mu_s)\d s,\quad t\geq 0,
\end{equation}
by virtue of \cite[Thm.~5.4.2]{FOT11}. 

	(ii) Let $u\in\D(\EE)$. 
	From the Markov property of ${(\mu_t)}_{t\geq 0}$ we conclude
	\begin{equation*}
		\E^\phi\big[M^{[u]}_{t+s}|\F_t\big]=\E^\phi\big[M^{[u]}_{t}+M^{[u]}_{s}\circ\theta_t|\F_t\big]=M^{[u]}_{t}+\E^{\mu_t}\big[M^{[u]}_{s}\big]=M^{[u]}_{t},\qquad t,s\geq 0,
	\end{equation*}
	$\P^\phi$-a.e.~for $\EE$-q.e.~$\phi\in E$. Hence, 
	\begin{equation*}
	\big({(M_t^{[u]})}_{t\geq 0},{(\scr F_t)}_{t\geq 0}, \P^\phi\big)
	\end{equation*} is a square-integrable martingale for $\EE$-q.e.~$\phi\in E$. The quadratic variation ${(\la M^{[u]}\ra_t)}_{t\geq 0}$ is in Revuz correspondence with 
	$$2\Gamma(u,u)(\phi)\Lam(\d\phi),$$ by virtue of \cite[Thm.~5.2.3]{FOT11}.
	Hence, by Remark \ref{rem:Revuz} and uniqueness of the Revuz measure,
	\begin{equation*}
		\la M^{[u]}\ra_t=2\int_0^t\Gamma(u,u)(\mu_s)\d s,\qquad t\geq 0.
	\end{equation*}
	Moreover, 
	\begin{equation*}
		\la M^{[u]},M^{[v]}\ra_t=2\int_0^t\Gamma(u,v)(\mu_s)\d s,\qquad t\geq 0,
	\end{equation*}
	for $u,v\in\D(\EE)$.
\end{rem}

In many cases, gradient-type Dirichlet forms on vector spaces are defined as an infinite sum of directional forms emerging from partial derivatives.
This is done, for example in \cite{AR90} and \cite[Chap.~2]{MR92}, from where we recall the notion of
$\LL$-admissibility.
Let $E$ and $\Lam$ be as above and in addition let $E$ be a locally convex, Hausdorff topological vector space.
The space of continuous linear functionals $E\to\RR$ is denoted by $E^*$.
A probability measure $m$ on $(\RR,\B(\RR))$ is said to satisfy the Hamza Condition if there exists a density $\rho_m:\RR\to[0,\infty)$
and an open set $U\subseteq \RR$ such that
\begin{equation}\label{eq:hamza}
	m(\d s)=\rho_m(s)\d s,\quad\rho_m^{-1}\in L_\loc(U)\quad\text{and}\quad m( U)=1.
\end{equation}
In this case, by the continuity of the embedding $L^2(\RR,m)\hookrightarrow L^1_\loc(U,m)$, we can define a weighted $(1,2)$-Sobolev space
\begin{equation*}
	H^{1,2}(m):=\big\{f\in H^{1,1}_\loc(U)\cap L^2(\RR,m):f'\in L^2(\RR,m)\big\}
\end{equation*}
with corresponding energy form and Sobolev-norm
\begin{equation*}
	{(f,g)}_{H^{1,2}(m)}:={\la f',g'\ra}_{L^2(\RR,m)},\qquad {\|f\|}_{H^{1,2}(m)}:=\big({\|f\|}^2_{L^2(\RR,m)}+ {(f,f)}_{H^{1,2}(m)}\big)^{\frac 1 2 },
\end{equation*}
for $f,g\in H^{1,2}(m)$.
For fixed $\phi\in E\setminus\{0\}$ and $\xi\in E^*$ with $\xi(\phi)=1$ we set
\begin{equation*}
	\pi_\phi:E\ni \eta\mapsto \eta -\xi(\eta)\phi,\quad     E_\phi:=\pi(E) \quad\text{and}\quad \nu_\phi:=\Lam\circ\pi_\phi^{-1}.
\end{equation*}
As shown in \cite[Chap.~10]{Dud02}, there exists a Markov kernel $K: E_\phi\times\B(\RR)\to[0,1]$ such that
\begin{equation*}
	\int_E f\d\Lam=\int_{E_\phi}\int_\RR f(\eta+s \phi)K(\eta,\d s)\nu_\phi(\d \eta),\ \ f\in\B_b(E).
\end{equation*}
We identify en element 
\begin{equation*}
{(u_\eta)}_{\eta\in E_\phi}\in L^2\big(\RR\times E_\phi, K(\eta,\d s)\nu_\phi(\d\eta)\big),
\qquad \RR\times E_\phi \ni(s,\eta)\mapsto u_\eta(s),
\end{equation*} 
with the unique element $u\in L^2(E,\Lam)$  such that
\begin{equation*}
\int_E uf\d\Lam=\int_{E_\phi}\int_\RR u_\eta(s)f(\eta+s \phi)K(\eta,\d s)\nu_\phi(\d \eta),\ \ f\in L^2(E,\LL).
\end{equation*}

\beg{defn}\label{def:admis} An element $\phi\in E\setminus\{0\}$ is called $\Lam$-admissible if there exist $\xi\in E^*$ and $K: E_\phi\times\B(\RR)\to[0,1]$ as above such that $K(\eta,\cdot)$ satisfies \eqref{eq:hamza} for $\nu_\phi$-a.e.~$\eta\in   E_\phi$.\end{defn}
For a $\LL$-admissible element $\phi\in E\setminus\{0\}$ we set
\begin{align*}
\D(\EE_\phi):=\Big\{{(u_\eta)}_{\eta\in  E_\phi}=u\in L^2(E,\LL):\ &u_\eta\in H^{1,2}(K(\eta,\cdot))
\ \text{for}\ \nu_\phi\text{-a.e.}\ \eta,\\ &\int_{E_\phi}{(u_\eta,u_\eta)}_{H^{1,2}(K(\eta,\cdot))}\nu_\phi(\d\eta)<\infty\Big\},
\end{align*}
\begin{equation*}
	\EE_\phi(u,v) :=\int_{E_\phi} {(u_\eta,v_\eta)}_{H^{1,2}(K(\eta,\cdot))}\nu_\phi(\d \eta).
\end{equation*}
This yields a well-defined Dirichlet form $(\EE_\phi,\D(\EE_\phi))$ in $L^2(E,\Lam)$, see \cite[Sect.~3]{AR90}.

\begin{exa}\label{exa:GaussDD}
Let $\Lam$ be a probability measure on $(E,\B(E))$ such that for each $\xi\in E^*$ the image measure $\Lam\circ\xi^{-1}$ on $\RR$ is Gaussian (including the case of a Dirac measure).
Then, $\Lam$ is called Gaussian measure on $E$ and there is a criterion for $\Lam$-admissibility via quasi-shift-invariance.

Let $\overline {E^{*}}^\Lam$ denote the closure of the set $\{\xi-\Lam(\xi):\xi\in E^*\}$ in $L^2(E,\Lam)$, where $\Lam(\xi):=\int_E\xi\d\Lam$. If $\phi\in E$ such that
\begin{equation}\label{eq:CM}
	\exists\, g_\phi\in \overline {E^{*}}^\Lam:\qquad \forall\, \xi\in E^*:\qquad \xi(\phi)=\int_E g_\phi(\tilde \phi) \big(\xi(\tilde \phi)-\Lam(\xi)\big)\Lam(\d\tilde \phi),
\end{equation}
then $\Lam$ and the image measure under the shift, $\Lam(\,\cdot\,-\phi)$ are absolutely continuous w.r.t.~each other.
The linear space of all $\phi\in E$ with the property  of \eqref{eq:CM} is called the Cameron--Martin space $\mathcal H_\LL$ of $\Lam$
and the corresponding Radon--Nikodym density is given by
\begin{equation*}
	\frac{\d\Lam(\,\cdot\,-\phi)}{\d\Lam}=\exp\big(g_\phi-\tfrac{1}{2}\|g_\phi\|_{L^2(E,\Lam)}\big),
\end{equation*}
see \cite[Chap.~2]{B98}.
The $\Lam$-admissibility of every non-zero element of the Cameron--Martin space $\mathcal H_\LL$ is shown in \cite[Prop.~4.2]{AR90}.
\end{exa}

To conclude the preliminaries, we give the definition  of an image form under a measurable transform of the state space.
Let $E$, $\Lam$ be as in the beginning of this section, $(\EE,\D(\EE))$ be a closed symmetric form in $L^2(E,\Lam)$ and $\Psi:E\to S$ be a measurable map into some measurable space $(S,\sigma)$.
The closed symmetric form $(\EE^{\text{im}},\D(\EE^{\text{im}}))$ in $L^2(S,\Lam\circ \Psi^{-1})$ which is defined
\begin{equation*}
\EE^{\text{im}}(u,v):=\EE(u\circ \Psi,v\circ \Psi )
\end{equation*}
and has the domain
\begin{equation*}
\D(\EE^{\text{im}}):=\big\{u\in L^2(S,\Lam\circ \Psi^{-1}):u\circ\Psi\in \D(\EE)\big\}
\end{equation*}
is called the image form  of $(\EE,\D(\EE))$ under $\Psi$.

\subsection{Induced diffusion processes on $\scr P_p$}\label{ssec:indT}
We fix $d\in\NN$ and denote the set of all probability measures on $\RR^d$ by $\scr P$.
For $p\in[1,\infty)$ the $p$-Wasserstein space \begin{equation*}
	\scr P_p:=\big\{\mu\in \scr P:\ \mu(|\cdot|^p) <\infty\big\}
\end{equation*}
of all probability measures with finite $p$-th moment
is a Polish space w.r.t.~the topology induced by the $p$-Wasserstein distance 
\begin{equation*}
	\WW_p(\mu,\nu):=\inf_{\pi\in\text{Coupl}(\mu,\nu)}\Big(\int_{\RR^d\times\RR^d}|x-y|^p\pi(\d x,\d y)\Big)^{\frac 1 p },\qquad \mu,\nu\in\scr P_p(\RR^d),
\end{equation*}
where the infimum runs over all couplings of $\mu$ and $\nu$.
In the following, we specify the basic existence result for a Markov process ${(R_t)}_{t\geq 0}$ on $\scr P_p$ which has a Gaussian-based invariant measure.
In \cite{BRW21} the intrinsic derivative on $\scr P_p$ and the class $C_b^1(\scr P_p)$ is introduced, in the spirit of \cite{AKR96, AKR98, R98}, as follows.
We use the convention that $\ff 1 0=\infty$.
\beg{defn}\label{D1} Let $p\in [1,\infty)$ and $\id(x)=x$ for $x\in\R^d$.
\beg{enumerate}\item[(1)] A continuous function
$f: \scr P_p \to\R$ is  called intrinsically differentiable if 
\begin{equation*}
L^p(\R^d\to\R^d,\mu)\ni \phi\mapsto D_\phi f(\mu):=\lim_{\vv\to 0} \ff{ f\big(\mu\circ(\id+\varepsilon\phi)^{-1}\big)
	- f(\mu)}{\vv}\in \R
\end{equation*}
is a well-defined, bounded linear functional for any $\mu\in \scr P_p$.
In this case, the intrinsic derivative $Df(\mu)\in L^{\ff p{p-1}}(\R^d\to\R^d,\mu)$
of $f$ at $\mu$ is the unique element  such that
\begin{equation*}
D_\phi f(\mu)=\int_{\RR^d}\<\phi, Df(\mu)\>\d\mu,\qquad \phi\in L^p(\R^d\to\R^d,\mu).
\end{equation*}
\item[(2)] An intrinsically differentiable function $f$ is called $L$-differentiable if 
\begin{equation*}
\lim_{\|\phi\|_{L^p(\RR^d\to\RR^d,\mu)}\downarrow 0} \ff{|f\big(\mu\circ(\id+\phi)^{-1}\big)-f(\mu)-D_\phi f(\mu)|}{\|\phi\|_{L^p(\RR^d\to\RR^d,\mu)}}=0
\end{equation*}
for each $\mu\in \scr P_p$.
The class $C_b^1(\scr P_p )$ contains all $L$-differentiable functions $f$ for which a bounded and continuous version 
$$\R^d\times\scr P_p \ni(x,\mu)\mapsto Df(x,\mu)\in \R^d$$
of the derivative exists in the sense that  $Df(\cdot,\mu)=Df(\mu)$ holds $\mu$-a.e.~for every $\mu\in\scr P_p$.
\end{enumerate}\end{defn}

\noindent Below, we define a Dirichlet form $\EE$ with state space $\scr P_p$, $p\in[1,2]$, and square-field operator of the type
\begin{align}\label{eq:Gammadef}
\Gamma(u,v)(\mu)&:={\la Du(\mu), Dv(\mu)\ra}_{L^2(\RR^d\to\RR^d,\gamma_\mu\mu)}\\
&=\int_{\RR^d}\la Du(\mu), Dv(\mu)\ra\gamma_\mu\d\mu,\qquad\mu\in\scr P_p,\,u,v\in C_b^1(\scr P_p),\nonumber
\end{align}
where $\gamma_\mu(x):=\gamma(x,\mu)$ and $\gamma:\RR^d\times\scr P_p\to(0,\infty)$ is measurable with $c^{-1}\leq\gamma(\cdot,\cdot)\leq c$ for some constant $c\in(0,\infty)$.
This  is analogous to the classical situation in which the state space $\BB$ is a Banach space, $\HH\subseteq \BB$ a densely embedded Hilbert space and 
thus $\BB^*\subseteq \HH^*=\HH$.
Then, $\Gamma(f,g)(x)=\la\nabla f(x),\nabla g(x)\ra_{\HH}$, $x\in\BB$, for suitable differentiable functions $f,g:\BB\to\RR$, is the square-field operator of standard gradient-type Dirichlet forms on $\BB$
(e.g.~see \cite{AR90}, \cite[Sect.~II.3]{MR92}).
Let 
\begin{equation*}\ppac:=\Big\{\mu\in\scr P_p:\text{ a probability density }\rho_\mu(x)=\frac{\d\mu}{\d x}\text{ exists}\Big\}.\end{equation*}
First, we formulate a condition on a reference probability measure on $L^p(\RR^d\to\RR^d,\lam)$ for some fixed $\lam\in\ppac$.
Subsequently, with the map in \eqref{eq:PsiLam} below, we consider its push-forward into $(\scr P_p,\B(\scr P_p))$. 
In this manner,  we can obtain a gradient-type Dirichlet form and diffusion process on $\scr P_p$.

\begin{enumerate}
	\item[$(C_1)$] Let $\lam\in\ppac $ for some $p\in[1,2]$ and $\D\subseteq L^p(\R^d\to\R^d,\ll) $ be a linear subspace,
	equipped with a locally convex Hausdorff topology which makes it a Lusin space, such that
	$\D$ is densely and continuously embedded into $L^p(\R^d\to\R^d,\ll)$.
	Moreover, let   $G_\D$ be a probability measure on $(\D,\B(\D))$   such that
	$L^2(\RR^d\to\R^d,\lam)$ has an  orthonormal basis ${\{\phi_k\}}_{k\in\mathbb N}$ consisting of
	$G_\D$-admissible elements (see Definition \ref{def:admis}).
\end{enumerate}

\begin{rem}\label{rem:loc}
	If  $ G_\D $ satisfies   condition $(C_1)$,
	then so does the measure $\frac{\eins_U(\phi)G_\D(\d\phi)}{G_\D(U)} $  for any   open set $U\subseteq \D$ with $G_\D(U)>0$.
	This is an immediate consequence of the fact that $\frac{\eins_D(s)m(\d s)}{m(D)}$ satisfies the Hamza Condition \eqref{eq:hamza}  if
	a probability $m$ on $(\RR,\B(\RR))$ does and  $  D\subseteq\RR$ is open with $m(D)>0$.
\end{rem}

Assuming ($C_1$), the push-forward of $G_\D$ under the map
\begin{equation}\label{eq:PsiLam}
	\Psi_\lam:\D\ni\phi\mapsto\lam\circ\phi^{-1}\in\scr P_p
\end{equation} 
yields a probability measure 
\begin{equation}\label{eq:Lam}\LL:=G_\D\circ\Psi_\lam^{-1}\end{equation} on
$(\scr P_p,\scr B(\scr P_p))$, suitable for the purpose of defining a gradient-type Dirichlet form.

Dirichlet forms and diffusion processes on $\scr P_p$ related to the intrinsic derivative 
are studied in \cite{RW22,RWW24} for such type of reference measures $\Lam$.
The next proposition sums up the relevant result in our context. 

\begin{prp}\label{prp:basicP} Assume $(C_1)$. Let $\Lam$ be as in \eqref{eq:Lam}, $\gamma:\RR^d\times\scr P\to (0,\infty)$ be measurable such that $c^{-1}\leq\gamma(\cdot,\cdot)\leq c$
	for some constant $c\in(0,\infty)$ and $\Gamma$ be as in \eqref{eq:Gammadef}.

	(i) The bilinear form $(\EE^{\gamma,\Lam},C_b^1(\scr P_p))$, defined 
	\begin{equation*}\EE^{\gamma,\Lam}(u,v):=\int_{\scr P_p}\Gamma(u,v)(\mu)\LL(\d\mu),\qquad u,v\in C_b^1(\scr P_p),
	\end{equation*} 
	has a closed extension in $L^2(\scr P_p,\LL)$. Its minimal closed extension yields a quasi-regular, strongly local Dirichlet form $(\EE^{\gamma,\Lam},\D(\EE^{\gamma,\Lam}))$.
	In particular, there exists a non-terminating diffusion $(\Omega, \scr F,$ $(\mu_t)_{t\geq 0}, (\P^\mu)_{\mu\in \scr P_p})$ on $\scr P_p$ which is associated with $\EE^{\gamma,\Lam}$.
	
	(ii) There exist $\phi_{\mu,k}\in L^2(\RR^d\to\RR^d,\mu)$, $k=\N$, $\mu\in\scr P_p$, such that $D_{\phi_{\mu,k}}u:\scr P_p\to\RR$ is measurable for $u\in C_b^1(\scr P_p)$ and 
	\begin{equation*}
		\Gamma(u,v)(\mu)=\sum_{k=1}^\infty D_{\gamma_\mu\phi_{\mu,k}} u(\mu)D_{\phi_{\mu,k}}v(\mu)
	\end{equation*}
	for $\Lam$-a.e.~$\mu\in\scr P_p,$ $u,v\in C_b^1(\scr P_p)$.
\end{prp}
\begin{proof}
	In the proof, we write $\EE:=\EE^{\gamma,\Lam}$ for short.
	
	(i) We may w.l.o.g.~assume $\gamma(\cdot,\cdot)=1$ regarding the claim of (i), which in view of the bounds for $\gamma$  does not affect the claimed properties.
	Let $X:=L^p(\RR^d\to\R^d,\lam)$ and $\id_{\D\to X}$ be the identification of an element in $\D$ with its $\lam$-class.
	For $k\in\NN$ and $\phi_k$ as in Condition $(C1)$ there exists a Dirichlet form $(\EE_{\phi_k},\D(\EE_{\phi_k}))$ on $\D$ corresponding to the directional derivative w.r.t.~$\phi_k$, as introduced at the end of Section \ref{ssec:general}.
	For simplicity, the image structure under $\id_{\D\to X}$, i.e.~image form and image measure on $X$, are again denoted by $\EE_{\phi_k}$ and $G_\D$.
	Let $C_b^1(X)$ denote the space of functions $f:X\to\RR$ whose Fr\'echet derivative $\na f$ yields a continuous and bounded map $X\to L^\infty(\RR^d\to\R^d,\lam)$.
	Then,
	$C_b^1(X)$ is contained in
	\begin{equation*}
		\D(\tilde \EE):=\Big\{f\in\bigcap_{k\in\NN}\D(\EE_{\phi_k}):\sum_{k\in\NN}\EE_{\phi_k}(f,f)<\infty\Big\}
	\end{equation*}
	since $\{\phi_k\}_{k\in\N}$ is an orthonormal basis of $L^2(\RR^d\to\RR^d,\lam)$.
	Moreover,
	\begin{equation*}
		\D(\tilde \EE)\times\D(\tilde\EE)\ni (f,g)\mapsto \sum_{k\in\NN}\EE_{\phi_k}(f,g)
	\end{equation*}
	is a closed symmetric form in $L^2(X,G_\D)$.

	Now, the proof can be completed by using the arguments   presented in the proofs of
	\cite[Thm.~3.2]{RW22} and \cite[Thm's.~2.1 \& 3.1]{RWW24}. First,
	$\Psi_\lam$ can be understood as a continuous map on $X$, because the assignment in \eqref{eq:PsiLam} respects $\lam$-classes.
	From \cite[Thm.~2.1]{BRW21}, it follows $u\circ\Psi_\lam\in C_b^1(X)$ for $u\in C_b^1(\scr P_p )$ with
	\begin{equation}\label{eq:ChainRIntrD}
		\na (u\circ \Psi_\lam)(\phi)=(Du\circ\Psi_\lam)\circ\phi,\qquad\phi\in X
	\end{equation}
	(see also \cite[Lem.~3.2]{RWW24}).
	This, in turn, implies that the image form of $(\tilde \EE,\D(\tilde \EE))$ under $\Psi_\lam$ is an extension of $(\EE, C_b^1(\scr P_p))$.
	Since such an extension, say $\EE$ with domain $\D(\EE)_{\text{large}}\subset L^2(\scr P_p ,\Lam)$, can be restricted to the topological closure $\overline{C_b^1(\scr P_p )}^{\EE_1}$ of $C_b^1(\scr P_p )$ in $\D(\EE)_{\text{large}}$ w.r.t.~$\EE_1^{1/2}$-norm, there exists a minimal closed symmetric form $(\EE,\D(\EE))=(\EE,\overline{C_b^1(\scr P_p )}^{\EE_1})$ in $L^2(\scr P_p ,\Lam)$ extending $(\EE, C_b^1(\scr P_p))$.
	The Markovian property of $(\EE,\D(\EE))$ is inherited via the relation in \eqref{eq:ChainRIntrD}
	and the (easy-to-prove) fact that $ \{\kappa\circ u:$ $u\in C_b^1(\scr P_p )\}$ $\subseteq C_b^1(\scr P_p )$ for $\kappa\in C_b^1(\RR)$.
	The strong local property and the existence of a square-field operator (i.e.~the fact that $\EE$ satisfies \eqref{eq:squareF}
	with $\Gamma(u,v)(\mu)=\la Du(\mu),Dv(\mu)\ra_{L^2(\RR^d\to\RR^d,\mu)}$
	is a consequence of the strong local property, respectively the product rule, of the
	gradient operator on $C_b^1(X)$ and again \eqref{eq:ChainRIntrD}.
	Finally, $(\EE,\D(\EE))$ is a quasi-regular Dirichlet form on $\scr P_p $ by the criterion in \cite[Thms.~2.1]{RWW24}, since $\D(\EE)$ contains all differentiable cylinder functions on $\scr P_p $,
	has a dense subset of continuous functions and 
	\begin{equation*}
		\EE(u,u)\leq \sup_{\mu\in\scr P_p}\|Du(\mu)\|^2_{L^{\ff p{p-1}}(\R^d\to\R^d,\mu)},\qquad u\in C_b^1(\scr P_p).
	\end{equation*}
	This concludes the proof of (i).
	
	(ii) Let $u,v\in C_b^1(\scr P_p)$ and $k\in\NN$. We refer to \cite[Sect.~3.2, in part.~Thm.~3.4 \& Exa.~3.7]{RWW24} for the existence of $\phi_{\mu,k}\in L^2(\RR^d\to\RR^d,\mu)$, $\mu\in\scr P_p$,
	as claimed such that 
	\begin{multline}\label{eq:ml}
		\int_{\scr P_p}\mu\big( \gamma_\mu \la Dv(\mu),\phi_{\mu,k}\ra\big)\mu\big(\la Dv(\mu),\phi_{\mu,k}\big\ra\big) \Lam(\d\mu)\\
		=\int_{\D}\lam \big( \gamma(\phi(\cdot),\Psi_\lam(\phi))\la \nabla (u\circ\Psi_\lam)(\phi),\phi_k\ra\big)
		\lam \big( \la \nabla (u\circ\Psi_\lam)(\phi),\phi_k\ra\big) G_\D(\d\phi).
	\end{multline}
	The right-hand side of this equation coincides with the image of a directional form $\EE_{\phi_k}$ as considered in (i) under $\Psi_\lam$, evaluated at $u,v$.
	\eqref{eq:ml} states that the image of a directional form can be represented in terms of a direction field  $\mu\mapsto \phi_{\mu,k}$ on $\scr P_p$.
	Summing up over $k\in\NN$, the right-hand side of \eqref{eq:ml} equals 
	\begin{equation*}
		\int_{\D}\lam\big(\gamma(\phi(\cdot),\Psi_\lam(\phi))\la \nabla (u\circ\Psi_\lam)(\phi),\nabla (v\circ\Psi_\lam)(\phi)\ra\big) G_\D(\d\phi)=\EE(u,v)
	\end{equation*}
	because of \eqref{eq:ChainRIntrD} and the fact that $\{\phi_k\}_{k\in\NN}$ is an orthonormal basis of $L^2(\RR^d\to\RR^d,\lam)$.
	The claim now follows easily by \eqref{eq:ml}.
\end{proof}

\subsection{Gaussian-based measures}\label{ssec:Gauss}

We introduce the class of Gaussian-based measures, which play an important role in the subsequent chapters (see Example \ref{exa:F2} and Remarks \ref{rem:Gauss2} \& \ref{rem:last} in particular).
From here on, the set $\scr D$ in Condition ($C_1$)  is always chosen as in \eqref{eq:DD1} below.
The starting point is  a Gaussian measure on $\D$, defined as in Example \ref{exa:GaussDD}.

\begin{rem}
	Let $\ll\in\ppac$ for some $p\in[1,2]$. The space
	\begin{equation}\label{eq:DD1}
		\D:=\Big\{\phi\in C^1(\RR^d,\RR^d):\ \sup_{\substack {x,y\in\RR^d\\x\ne y}}\ff{|\phi(x)-\phi(y)|}{|x-y|}=:\|\na\phi\|_\infty<\infty\Big\}
		\end{equation}
		with the metric
		\begin{equation*}
		 d_\D(\phi_1,\phi_2):=|\phi_1-\phi_2|(0) + \big\|\na(\phi_1-\phi_2)\big\|_\infty,\qquad \phi_1,\phi_2\in\D,\nonumber
	\end{equation*}
	is complete and separable.
	Moreover, $(C_1)$ is satisfied given any Gaussian measure $G_\D$ on $(\D,\B(\D))$
	such that the inclusion 
	\begin{equation*}
	\mathcal H_{G_\D}\cap L^2(\RR^d\to\R^d,\lam)\subseteq L^2(\RR^d\to\R^d,\lam)
	\end{equation*}
	is dense ($\mathcal H_{G_\D}$ denotes the Cameron-Martin space).
\end{rem}

We denote the Jacobian matrix of a function $\phi\in \D$ by $\na\phi$. Since 
\begin{equation*}
	\sup_{z\in\RR^d}\,\sup_{h\in\RR^d\setminus\{0\}}\,\frac{|\na\phi(z)h|}{|h|}=\sup_{\substack {x,y\in\RR^d\\x\ne y}}\,\ff{|\phi(x)-\phi(y)|}{|x-y|},
\end{equation*}
we may also read $\big\|\na\phi \big\|_\infty$ as the supremum norm of the operator norm of $\na\phi$.

\begin{exa}\label{exa:GaussD}
	Let  $\lam\in\ppac$, $p\in[1,2]$.
A Gaussian measure $G_\D$ as in the above remark can be constructed in a natural, starting with any sequence ${\{\phi_k\}}_{k\in\mathbb N}\subset C_b^1(\RR^d,\RR^d)$ which is an
orthonormal basis of $L^2(\RR^d\to\RR^d,\lam)$, as demonstrated in the following.
	 We set
	\begin{equation}\label{eq:ak}
		a_k:=\max\big\{2^k\|\na\phi_k\|^2_\infty,1\big\}
	\end{equation}
	for $k\in\N$.
	Defining $T_{\lam,2}:=L^2(\RR^d\to\RR^d,\lam)$ and
	\begin{align*}
		\HH&:=\Big\{\phi\in T_{\lam,2}:\sum_{k=1}^\infty a_k{\la\phi,\phi_k\ra}_{T_{\lam,2}}^2<\infty\Big\},\\
		{\la\phi_1,\phi_2\ra}_{\HH}&:=\sum_{k=1}^\infty a_k{\la\phi_1,\phi_k\ra}_{T_{\lam,2}}{\la\phi_2,\phi_k\ra}_{T_{\lam,2}},\quad\phi_1,\phi_2\in\HH,
	\end{align*}
	yields a Hilbert space with $\HH\subset L^2(\RR^d\to\RR^d,\lam)\cap \D$, because of the estimate
	\begin{equation*}
		\|\na\phi\|_\infty\leq\sum_{k=1}^\infty{\la\phi,\phi_k\ra}_{T_{\lam,2}}\|\na \phi_k\|_\infty
		\leq \|\phi\|_\HH\Big(\sum_{k=1}^\infty \tfrac{\|\na\phi_k\|_\infty^2}{a_k}\Big)^{1/2}\leq\|\phi\|_\HH
	\end{equation*}
	for $\phi\in\HH$.
	If we choose a sequence ${(b_k)}_{k\in\NN}$ such that $\sum_{k\in\NN}\frac{a_k}{b_k}<\infty$, then the measure
	\begin{equation*}
		G_\D(\d \phi):= \prod_{k=1}^\infty m_k(\d \<\phi_k,\phi\>_{T_{\lam,2}})\quad \text{with}\quad m_k(\d r):=  \Big(\ff{b_k}{2\pi}\Big)^{\ff 1 2} \exp\Big(-\ff{b_kr^2}{2}\Big)\d r\end{equation*}
	guarantees that $(C_1)$ holds true, as each $\phi_k$ is admissible and $G_\D$ is a Gaussian measure on $\HH$, in particular on $\D$. The latter is true, since $G_\D$ can be rewritten as
	\begin{equation*}
		G_\D(\d \phi)= \prod_{k=1}^\infty \tilde m_k(\d \<\tilde \phi_k,\phi\>_{T_{\lam,2}})\quad\text{with}\quad
	\tilde m_k(\d r):=  \Big(\ff{b_k}{2a_k\pi}\Big)^{\ff 1 2} \exp\Big(-\ff{b_kr^2}{2a_k}\Big)\d r
	\end{equation*}
	in terms of the re-scaled basis elements $\tilde \phi_k:=(\alpha_k)^{-1/2}\phi_k$, $k\in\N$, which are normalized in $\HH$.
\end{exa}

We consider the set $\DDp$ as in \eqref{eq:DDp} and a Gaussian measure $G_\D$ on $(\D,\B(\D))$ with full topological support and Condition $(C_1)$. 
The existence of such a measure is shown in Example \ref{exa:GaussD} above. Rather than taking $G_\D\circ\Psi_\lam^{-1}$ itself as the reference measure on $\scr P_p$, we would like to 
work with the image measure of the restriction of $G_\D$ to the set $\DDp$ under $\Psi_\lam$.
The reason is that the relevant energy functional $W_F$ discussed in the introduction is firstly defined almost surely w.r.t.~the latter and secondly $e^{-W_F}$ becomes integrable.
This allows us to define a perturbed Gaussian-based measure with  a density proportional to  $e^{-W_F}$, as this is the subject of Sections \ref{sec:EF} \& \ref{sec:Loc}.
\begin{rem}\label{rem:DDnew}
The set $\DDp$ as in \eqref{eq:DDp} coincides with the set of all continuously differentiable bi-Lipschitz functions $\phi:\RR^d\to\RR^d$ with $\det[\na\phi]\neq 0$ everywhere.
Indeed, if $\phi$ is bi-Lipschitz, then it is easy to show that $\phi(\RR^d)$ is a closed and also an open set (invariance of domain theorem),
and by the inverse mapping theorem $\phi^{-1}:\RR^d\to\RR^d$ is continuously differentiable.

Moreover, $\DDp$ is an open subset of $(\D, d_\D)$ by the following argument.
Let $\phi_0\in\DDp$. We set $(\|\nn \phi_0\|_\infty+\|\nn (\phi_0^{-1})\|_\infty)=:C\in [1,\infty)$ and show that 
\begin{equation*}
\phi\in\D:\qquad 	d_\D(\phi,\phi_0)<(d+1)^{-1}(C+1)^{1-2d}
\end{equation*} implies $\phi\in\DDp$. First, using $\|\na(\phi-\phi_0)\|_\infty\leq d_\D(\phi,\phi_0)\leq \min\{\tfrac 1  4,\tfrac {1}{2C}\}$, we estimate
\begin{align*}
	\frac{1}{2C}|x-y|&\leq \frac 1 C |x-y|- \|\na(\phi-\phi_0)\|_\infty|x-y|\\
	&\leq |\phi_0(x)-\phi_0(y)|-|(\phi-\phi_0)(x)-(\phi-\phi_0)(y)|\\
	&\leq |\phi(x)-\phi(y)|\\
	&\leq |\phi_0(x)-\phi_0(y)|+|(\phi-\phi_0)(x)-(\phi-\phi_0)(y)| \\
	&\leq C|x-y|+\|\na(\phi-\phi_0)\|_\infty|x-y|\leq (C+\tfrac 1 4)|x-y|
\end{align*}
proving that $\phi$ is bi-Lipschitz with $\|\nn \phi\|_\infty\leq C+\tfrac 1 4 $ and $\|\nn (\phi^{-1})\|_\infty\leq 2C$.
Next, we use Hadamard's inequality
\begin{equation}\label{eq:Hadamard}
	\inf_{x\in\RR^d}|\det[\na\phi_0(x)]|=\inf_{x\in\RR^d}\Big({\big|\det\big[\na{(\phi_0^{-1})}(\phi_0(x))\big]\big|}^{-1}\Big)\geq 
	\|\na{(\phi_0^{-1})}\|_\infty^{-d}
\end{equation}
and a well-known perturbation result for the determinant (see e.g.~\cite[Chap.~5]{Bhat})
\begin{equation*}
	\sup_{x\in\RR^d}\big|\det[\na \phi(x)]-\det[\na\phi_0(x) ]\big|\leq d \big(\max\{\|\na\phi \|_\infty,\|\na\phi_0\|_\infty\}\big)^{d-1}\big \|\na(\phi-\phi_0)\big\|_\infty
\end{equation*}
to obtain
\begin{equation*}
	\inf_{x\in\RR^d}|\det[\na\phi(x)]|\geq \|\na{(\phi_0^{-1})}\|_\infty^{-d}-d (C+1)^{d-1}d_\D(\phi,\phi_0)>(C+1)^{-d}(1-\tfrac{d}{d+1})
\end{equation*}
and hence $\phi\in\DDp$.
Since we have shown that $\DDp$ is a non-empty (it contains the identity map) and open subset of $\D$, we have $G_\D(\DDp)>0$ and Remark \ref{rem:loc} applies.
\end{rem}

The composition makes $(\DDp,\circ)$ a group. Hence, the sets
\begin{equation}\label{eq:equivR}
	[\ll]^\sim:= \Psi_\ll(\DDp),\qquad\lam\in\ppac,
	 \end{equation}
are equivalence classes on $\ppac$.
The definition of $\Lam$ depends on a fixed choice of $\lam\in\ppac$. 
The role of that $\lam$ is discussed in Remark \ref{rem:Gauss} below. The essential observation is that,
instead of $\lam$ and \eqref{eq:PsiLam}, we may take any element $\mu\in[\lam]^\sim$.
Then, $\Lam$ can equivalently be represented as the push-forward of some Gaussian measure (which is a linear transform of the Gaussian $G_\D$) under \begin{equation*}
	\Psi_\mu:\D\ni\phi\mapsto\mu\circ\phi^{-1}\in\scr P_p.
\end{equation*}

\begin{rem}\label{rem:Gauss}
	Let $(\D,d_\D)$ be as in \eqref{eq:DD1} and $\scr N_\D$ be the set of all Gaussian measures on $(\D,\B(\D))$ with full topological support.
	
	(i) Given $\phi\in\DDp$ the vector space isomorphism
	\begin{equation*}
		K_\phi:\D\ni\tilde \phi\mapsto\tilde\phi\circ\phi\in\D
	\end{equation*}
	defines permutation on  $\scr N_\D$ through the assignment $$G_\D\mapsto G_\D\circ K_\phi^{-1}.$$
	By the elementary relation
	\begin{equation*}
		(G_\D\circ K_\phi^{-1})\circ\Psi_\lam^{-1}=G_\D\circ\Psi_{\Psi_\lam(\phi)}^{-1}
	\end{equation*}
	the family $\{G_\D\circ\Psi_\lam^{-1}:G_\D\in\scr N_\D\}$ coincides with $\{G_\D\circ\Psi_\mu^{-1}:G_\D\in\scr N_\D\}$
	if $\mu,\lam$ $\in\ppac$ are equivalent in the sense of \eqref{eq:equivR}. This family of measures is henceforth denoted by $\scr G_{[\lam]}$.
	
	(ii) Since $K_\phi(\DDp)=\DDp$ for $\phi\in\DDp$, the same reasoning applies to the family
	\begin{equation*}
		\scr G^1_{[\lam]}:=\Big\{G_\D\circ\Psi_\lam^{-1}: G_\D(\d\phi)=\frac{\eins_{\DDp}(\phi) G_\D'(\d\phi)}{ G_\D'({\DDp})}\text{ for some } G_\D'\in\scr N_\D\Big\}
	\end{equation*}
	for $\lam\in\ppac$. We note that $\Lam(\ppac)=1$ for $\Lam\in \scr G^1_{[\lam]}$, because every function in $\DDp$ is a diffeomorphism $\RR^d\to\RR^d$.
	
	(iii) Let $p=2$.  $\D$ is densely contained in $L^2(\RR^d\to\R^d,\lam)$ for all $\lam\in\ppac$ and so is the Cameron--Martin space $\mathcal H_{G_\D}$ for every $G_\D\in \scr N_\D$.
	Hence, Condition $(C_1)$ is always satisfied if $\lam\in\pac_2$ and $G_\D\in \scr N_\D$. The same is true for  
	$$G_\D:=\frac{\eins_{\DDp} G_\D'}{ G_\D'({\DDp})},\qquad G_\D'\in\scr N_\D,$$
	by Remark \ref{rem:loc} together with Remark \ref{rem:DDnew}.
\end{rem}

In the upcoming Section \ref{sec:EF},
the measure $G_\D(\d\phi)=\frac{\eins_{\DDp}(\phi) G_\D'(\d\phi)}{ G_\D'({\DDp})}\text{ for } G_\D'\in\scr N_\D$ (as in Remark \ref{rem:Gauss}) is the most important example
for which the main assumption (see $(C_3)$ below) and thereby the main result (Theorem \ref{thm:main}) can be verified.
Accordingly, $\scr G^1_{[\lam]}$ is a class of \lq good\rq\, reference measures on $\scr P_p$ for our purpose. Before we conclude this section, we prove a result on the topological support of $\Lam$,
given that $\Lam\in \scr G_{[\lam]}$ or $\Lam\in \scr G^1_{[\lam]}$ in the most relevant case $p=2$.

\begin{lem}\label{lem:W2} Let $\lam\in\twopac$. Any measure $\Lam\in\scr G_{[\lam]}\cup \scr G_{[\lam]}^1$ has full topological support on $\scr P_2$.
	\end{lem}

 \begin{proof}  
It suffices to show that $\Psi_\lam(\D)$ and $\Psi_\lam(\DDp)$ are dense in $\scr P_2 $.
	If so, for any open neighborhood $U$ of an element $\mu\in\scr P_2 $, the sets  $\Psi_\lam^{-1}(U)$ and $\Psi^{-1}_\lam(U)\cap\DDp$ are non-empty open in $\D$ and hence
	they are assigned a positive probability w.r.t.~any non-degenerate Gaussian on $\D$. The density of $\Psi_\lam(\D)$ is trivial, since $\D$ is dense in $L^2(\RR^d\to\R^d,\lam)$
	and for each $\mu\in\scr P_2 $ there is a measurable transport map $\phi\in L^2(\RR^d\to\RR^d,\lam)$ such that $\lam\circ\phi^{-1}=\mu$. Then,
	\begin{equation*}
		\limsup_{n\to\infty}\W_2(\lam\circ\phi_n^{-1},\mu)\leq \limsup_{n\to\infty}{\|\phi_n-\phi\|}_{L^2(\RR^d\to\RR^d,\lam)}=0
	\end{equation*}
	if ${(\phi_n)}_n\subset \D$  is a sequence approximating $\phi$ in $L^2(\RR^d\to\RR^d,\lam)$.
	
	In the case of $\Psi_\lam(\DDp)$, we argue as follows. For an arbitrary element $\mu\in\scr P_2 $ we can choose an optimal pair $(\varphi,\varphi^c)$ such that
	\begin{equation*}
		\lam(\varphi)+\mu(\varphi^c)=\sup_{(f,g)}(\lam(f)+\mu(g))= \W_2(\mu,\ll),
	\end{equation*}
	where the supremum runs over all $(f,g)\in C_b$ with $f(x)+g(y)\leq|x-y|^2$ for $x,y\in\RR^d$ and $\varphi^c(y):=\inf_{x\in\RR^d}(|x-y|^2-\varphi(x))$.
	Then, we know from the theory of optimal transport that $|\cdot|^2-\varphi$ is convex,
	$\varphi$ has an approximate differential $\na \varphi$ and
	\begin{equation*}
		\phi:\RR^d\ni x\mapsto x-\tfrac{1}{2}\na \varphi(x)\in\RR^d
	\end{equation*}
	is an optimal transport map with $\lam\circ\phi^{-1}=\mu$, see \cite[Thm.~6.2.4]{AGS05}.  The claim follows if we can approximate $\phi$ by elements of $\DDp$ in $L^2(\RR^d\to\RR^d,\lam)$.
	
		Since $|\cdot|^2-\varphi$ is convex, 
by mollyfying we can find a sequence ${\{\varphi_n\}}_{n}\in C^1(\RR^d)$ such that, for fixed $n\geq 1$, the minimal eigenvalue of the Hessian matrix of $(1+\tfrac {1} {n}) |\cdot|^2- \varphi_n$
	is bounded from below by a positive constant, i.e.
		\begin{equation*}
		\inf_{x\in\RR^d}\,\inf_{z\in\RR^d,\,|z|=1}\,z^\tr\big(2+\tfrac {2} {n} - \na\na \varphi_n(x)\big)z >0,
	\end{equation*}
	and $\na\varphi_n\to \na\varphi$ in $L^2(\RR^d\to \RR^d,\lam)$
	This implies that the function
		\begin{equation*}
		\phi_n(x):= (1+\tfrac{1}{n})x-\tfrac{1}{2}\na \varphi_n(x),\qquad x\in\RR^d,
	\end{equation*}
	is injective and its inverse mapping $\phi_n^{-1}:\phi_n(\RR^d)\to \RR^d$ is Lipschitz continuous.
	For any Cauchy sequence ${(y_l)}_l\subset \phi_n(\RR^d)$ the image ${(x_l)}_l:={(\phi_n^{-1}(y_l))}_l$ under $\phi_n^{-1}$ is also Cauchy
	and has a limit $x_0:=\lim_{l\to\infty}\phi_n^{-1}(y_l)$. Then, $\phi_n(x_0)=\lim_{l\to\infty}y_l$ and hence $\phi_n(\RR^d)$ is a closed set. By the invariance of domain theorem
	$\phi_n(\RR^d)$ is also open and so,  $\phi_n(\RR^d)=\RR^d$.

	For $m,n\geq 1$, the map  $\RR^d\ni x\mapsto \phi_n^{-1}(x)+\tfrac x m\in\RR^d$
	is differentiable and bi-Lipschitz, since 
	\begin{equation*}
		0<\inf_{x\in\RR^d}\,\inf_{z\in\RR^d,\,|z|=1}\,z^\tr\big(\na(\phi_n^{-1})(x)+\tfrac 1 m\big)z < \sup_{x\in\RR^d}\,\sup_{z\in\RR^d,\,|z|=1}\,z^\tr\big(\na(\phi_n^{-1})(x)+\tfrac 1 m\big)z<\infty.
	\end{equation*}
	Then, also its inverse mapping
	\begin{equation*}
		\phi_{n,m}:= \big(\phi_n^{-1} +\tfrac {\id} m  \big)^{-1}:\RR^d\to\RR^d
	\end{equation*} is differentiable, bi-Lipschitz and hence $\phi_{n,m}\in \DDp$. 
	Now, the claimed density of $\Psi_\lam(\DDp)$ holds true, since 
	\begin{equation*}	
	\Psi_\lam(\phi_{n,m})=\lam\circ\big(\phi_n^{-1}+\tfrac{\id}{m}\big)
	\overset{m\to\infty}{\longrightarrow}\lam\circ\phi_n^{-1}\overset{n\to\infty}{\longrightarrow}\lam\circ\phi=\mu
	\end{equation*}
	 w.r.t.~$\W_2$.
\end{proof}

\section{Perturbation by energy functionals}\label{sec:EF}
In this section, let $\lam\in\ppac$, $p\in[1,2]$, $(\D,d_\D)$ be as in \eqref{eq:DD1} and $\DDp$ be as in \eqref{eq:DDp}. The map $\Psi_\lam:\D\to\scr P_p$ is defined in \eqref{eq:PsiLam}.
We are interested in certain perturbations for the diffusion process ${(R_t)}_{t\ge 0}$ with Dirichlet form $\EE^{\gamma,\Lam}$ given in Proposition \ref{prp:basicP},
which arise from modifying $\Lambda$.
The invariant measure $\Lam$ of ${(R_t)}_{t\ge 0}$ is now multiplied by a factor proportional to $\e^{-W_F}$ (see \eqref{eq:LW}, \eqref{eq:W})
and we obtain a new process associated with a Dirichlet form $\EE^F$ which has the square-field operator $\Gamma$ of $\EE^{\gamma,\Lam}$, but a perturbed reference measure $\Lambda_F$.
The energy functional $W_F$ is of the type  \eqref{eq:W} and thus takes values smaller infinity 
only on a certain domain within the set of absolute continuous measures. Conditions $(C_2)$ and $(C_3)$ below make sure the measure $\Lam_F$
in \eqref{eq:LW} is well-defined and a stochastic quantization exists, i.e.~$\Lam_F$ is the stationary distribution of a gradient diffusion process $(\mu_t)_{t \ge 0}$ (see Theorem \ref{thm:main} below).
Consequently, as stated in Corollary \ref{cor:SFG} below, ${(\mu_t)}_{t\ge 0}$ is a martingale solution to \eqref{eq:SGF0} in the sense of \eqref{eq:M},
in case $W_F\in\D(\EE^{\gamma,\Lam})$ such that $\e^{-W_F}\in\D(\EE^{\gamma,\Lam})\cap L^\infty(\Lam)$.

 We assume that $\ll$ and $F$ satisfy:
\begin{enumerate}\item[$(C_2)$] $\ll\in  \ppac $,   $F:\RR^{d}\times[0,\infty)\to\RR$ is measurable such that  $\int_{\RR^d}\bar F_\aa(x)\d x<\infty$
	for any $\aa\in (1,\infty)$, where
	\begin{equation*}
	\bar F_\aa(x):=	  \sup\Big\{\big| F\big(y,t\rr_\ll(x)\big)\big|:\ (y,t)\in \R^d\times\R,\  |y|\le  \aa(1+ |x|), \  t\in [\aa^{-1},\aa] \Big\}.
	\end{equation*}
\end{enumerate} 
\begin{rem}\label{rem:rhomu}
(i) Every $\phi\in\DDp$ is a $C^1$-diffeomorphism $\RR^d\to\RR^d$ with
\begin{equation*}
\inf_{x\in\R^d}\big|{\rm det}[ \nn \phi(x)]\big|>0,
\end{equation*}
see \eqref{eq:Hadamard}. For $\lam\in\ppac$ the transformation rule yields
\begin{equation*}
	\int_{\RR^d} (f\circ\phi)\lam(\d x)=\int_{\RR^d}\frac{f(x)\rho_\lam(\phi^{-1}(x))\d x}{|\det[\na\phi(\phi^{-1}(x))]|}
\end{equation*}
for bounded measurable $f:\RR^d\to\RR$. Hence, if $\mu=\Psi_\lam(\phi)$, then
\begin{equation*}
	\rho_\mu(x):=\frac{\mu(\d x)}{\d x}=\frac{\rho_\lam(\phi^{-1}(x))}{|\det[\na\phi(\phi^{-1}(x))]|}.
\end{equation*}

 (ii) With Condition $(C_2)$ the density $\e^{-W_F}$ is strictly positive on $\Psi_\lam(\DDp)$. Indeed, for $\phi\in\DDp$,
\begin{equation*}
	\int_{\RR^d}\big|F\big(\phi(x),|\det[\na\phi(x)]|^{-1}\rr_\ll(x)\big) \det[\na\phi(x)]\big|\d x<\infty.
\end{equation*}
Thus, if $\mu=\lam\circ\phi^{-1}$, then $\e^{-W_F(\mu)}>0$ due to
\begin{equation}\label{eq:trafo}
	W_F(\mu)= \int_{\RR^d}F(x,\rho_\mu(x))\d x
	=\int_{\RR^d}F\big(\phi(x),|\det[\na\phi(x)]|^{-1}\rr_\ll(x)\big) \big|\det[\na\phi(x)]\big|\d x.\nonumber
\end{equation}
\end{rem}
We present a relevant example for a function $F$, which is further discussed in Example \ref{exa:F2} and Corollary \ref{cor:locSGH} below. 
\begin{exa}\label{exa:F}
	If $\lam\in \ppac $ with $\ll(|\ln \rho_\lam|)<\infty$,
	then  $(C_2)$ is satisfied  for
	\begin{equation*}
		F(x,s):=sV(x)+\int_0^s\int_1^t\frac{q(r)}{r}\d rd t,\quad x\in\RR^d,\,s\in[0,\infty),
	\end{equation*}
	where $V\in C(\RR^d)$  and $q:(0,\infty)\to(0,\infty)$ is measurable such that
	$$|V(\cdot)|\leq c(1+|\cdot|),\qquad c^{-1}\leq q(\cdot)\leq c,$$
	for some constant $c\in(0,\infty)$. Indeed, we can find a constant $\tilde c\in(0,\infty)$ such that
	\begin{equation*}
		|F(x,s)|\leq \tilde c s(1+|x|+|\ln(s)|),\quad x\in\RR^d, \,s>0,
	\end{equation*}
	by which $\int_{\RR^d}\bar F_\aa(x)\d x<\infty$ for $\alpha\in(1,\infty)$ is obvious.
\end{exa}
\begin{rem}\label{rem:gPME}
	Let $\beta:\RR\to\RR$, $b:\RR\to(0,\infty)$, $\Phi\in \RR^d\to\RR$ be as in \cite[Hypothesis 1]{RR23}, in particular
	\begin{align*}
		&\beta\in C^1(\RR), \quad\beta(0)=0,\quad c^{-1}\leq \beta'\leq c,\\
		&b\in C_b(\RR)\cap C^1(\RR),\quad c^{-1}\leq b,\\
		&\Phi\in C^1(\RR^d),\quad \na\phi\in C_b(\RR^d,\RR^d),
	\end{align*}
	 for some constant $c\in(0,\infty)$ and 
	\begin{equation}\label{eq:FRR}
		F(x,s):=s\Phi(x)+\int_0^s\int_1^t\frac{\beta'(r)}{rb(r)}\d rd t,\quad x\in\RR^d,\,s\in[0,\infty).
	\end{equation}
	By \cite[Thm.'s~2, 3.7 \& 3.8]{RR23} there is a one-to-one correspondence between the probability solutions to
	the generalized porous media equation \eqref{eq:gPME}
	and the gradient flow in \eqref {eq:I1}.
	If $\lam\in \ppac $ with $\ll(\ln| \rho_\lam|)<\infty$, then $(C_2)$ holds for $F$ in \eqref{eq:FRR}, 
	as this is a special case of Example \ref{exa:F}. In Example \ref{exa:F2} below, where we continue in this setting, we 
	choose $\Lam$ from the family $\scr G_{[\lam]}^1$ of Gaussian-based measures (see Remark \ref{rem:Gauss} (ii)) with $p=2$ and 
	give an existence statement for a diffusion ${(\mu_t)}_{t\geq 0}$ on $\scr P_2$ which has invariant measure $\Lam_F$ as defined in \eqref{eq:LW}.
	The representation of  ${(\mu_t)}_{t\geq 0}$ as a stochastic gradient flow is discussed in Section \ref{sec:Loc}.
\end{rem}

The next condition makes sure the bilinear form in \eqref{eq:Dform} is well-defined on $C_b^1(\scr P_p)$ and has a closed extension in $L^2(\scr P_p,\Lam_F)$ which yields a diffusion process on $\scr P_p$.
\begin{enumerate} 
	\item[$(C_3)$]  
	For given $\lam\in\ppac$ together with a probability measure $G_\D$ on $(\D,\B(\D))$ and a function $F:\RR^{d}\times(0,\infty)\to\RR$ we assume:
	\begin{itemize}
		\item $(C_1)$, $(C_2)$ hold true.
\item $G_\D(\DDp)=1.$
\item$\displaystyle Z_F:=\int_{\DDp}\e^{-W_F(\ll\circ \phi^{-1})}G_\D(\d\phi)<\infty.$
\end{itemize}
\end{enumerate}

\begin{rem}
With Condition $(C_3)$ and $\Lam:=G_\D\circ\Psi_\lam^{-1}$, it follows $\Lam(\ppac)=1$ and the function $\e^{-W_F}$ is strictly positive, $\Lam$-almost surely, because of \eqref{eq:trafo}.
Moreover, $\int_{\scr P_p } \e^{-W_F(\mu)}\LL(\d\mu)=\int_{\D_1}\e^{-W_F(\ll\circ \phi^{-1})}G_\D(\d\phi)$ and so, \eqref{eq:LW}
defines a probability measure on $\ppac$. 
\end{rem}

In Theorem \ref{thm:main} below, we assume $(C_3)$
and consider the bilinear form 
\beq\label{eq:EF}
\EE^F(u,v):=\int_{\scr P_p} \GG(u,v)\d\LL_F,\ \ \ u,v\in C_b^1(\scr P_p),
\end{equation}
with $\Gamma$,  $\Lam_F$
as in \eqref{eq:Gammadef} and \eqref{eq:LW}.

\begin{thm}\label{thm:main} 
$(\EE^F, C_b^1(\scr P_p))$ has a minimal closed extension in $L^2(\scr P_p,\LL_F)$ which is  a quasi-regular, strongly local Dirichlet form.
In particular, there exists a non-terminating diffusion $(\Omega, \scr F,(\mu_t)_{t\geq 0},$ $(\P^\mu)_{\mu\in \scr P_p})$ on $\scr P_p $ which is  associated with
the minimal closed extension, which we denote by $(\EE^F,\D(\EE^F))$.
\end{thm}
\begin{proof} By Proposition \ref{prp:basicP}, it suffices to verify condition $(C_1)$ for the measure
	$  (\e^{-W_F}\circ\Psi_\lam)\d G_\D$ replacing $G_\D$.
	
	We observe, that the Hamza Condition \eqref{eq:hamza} for a probability measure $m$ on $(\RR,\B(\RR))$ is inherited to $\varrho m$, if $\varrho :\RR\to [0,\infty)$ is measurable,
	$\int_\RR \varrho\d m=1$ and
	$$m\big(\{t\in\RR:\varrho  \text{\ is\ continuous\ on }\  (t-\varepsilon,t+\varepsilon) \text{\ for\ some }\varepsilon>0\}\big)=1.$$
	So, inferring that $\DDp$ is an open set in $\D$ (see Remark \ref{rem:DDnew}), the claim follows once we show that $  \e^{-W_F}\circ\Psi_\lam$ is continuous on $\DDp$.
	Let $\xi\in\DDp$, i.e.$(\|\nn \xi\|_\infty+\|\nn (\xi^{-1})\|_\infty)=:C\in[1,\infty)$, and
	\begin{equation*}
		B(\xi):=\Big\{\phi\in\D: d_\D(\phi,\xi)< (d+1)^{-1}(C+1)^{1-2d}   \Big\}.
	\end{equation*}
	Then, by Remark \ref{rem:DDnew}, $B(\xi)\subset\DDp$  with 
	\begin{equation}\label{eq:SpecLB}
		\|\na\phi\|_\infty\leq C+\tfrac 1 4\quad\text{and}\quad \inf_{x\in \RR^d}|\det[\na\phi(x)]| > (C+1)^{-d}(1-\tfrac{d}{d+1})\qquad \text{for }\phi\in B(\xi).
	\end{equation}
	Analogously to \eqref{eq:trafo}, the transformation formula yields
	\begin{equation}\label{eq:trafo2}
		\int_{\RR^d}F(x,\rho_{\lam\circ\phi^{-1}}(x))\d x
		=\int_{\RR^d} F\big(\phi(x),|\det[\na\phi(x)]|^{-1}\rho_\ll(x)\big)  |\det[\na\phi(x)]| \d x
	\end{equation}
	for $\phi\in B(\xi)$.
	Due to $(C_2)$ and \eqref{eq:SpecLB} we can find a $\d x$-integrable function which dominates the integrand of the right-hand side of \eqref{eq:trafo2}, uniformly for $\phi\in B(\xi)$.
	Hence, using Lebesgue's dominated convergence and \eqref{eq:trafo2}, if $\{\phi_n\}_{n\in\NN}\subset \DDp$, $\phi\in B(\xi)$, such that $\lim_{n\to\infty}d_\D(\phi_n,\phi)=0$, then
	\begin{align*}
		\lim_{n\to\infty}(W_F\circ\Psi_\lam)(\phi_n)&=\lim_{n\to\infty}\int_{\RR^d} F\big(\phi_n(x),|\det[\na\phi_n(x)]|^{-1}\rho_\ll(x)\big)  |\det[\na\phi_n(x)]| \d x\\
		&=\int_{\RR^d} F\big(\phi(x),|\det[\na\phi(x)]|^{-1}\rho_\ll(x)\big)  |\det[\na\phi(x)]| \d x\\
		&=\int_{\RR^d}F(x,\rho_{\lam\circ\phi^{-1}}(x))\d x=(W_F\circ\Psi_\lam)(\phi).
	\end{align*}
	This implies continuity of $\e^{-W_F}\circ\Psi_\lam$ on $\DDp$ and concludes the proof.
\end{proof}

The next example shows that Theorem \ref{thm:main} is applicable for $F$ as in Example \ref{exa:F} and $\Lam\in \scr G_{[\lam]}^1$ as defined in Remark \ref{rem:Gauss} (ii).
We choose $p=2$ to ensure Condition $(C_1)$ through the argument of  Remark \ref{rem:Gauss} (iii).

\begin{exa}\label{exa:F2}
	Let $\lam\in\twopac$, $\ll(|\ln \rho_\lam|)<\infty$, $F:\RR^{d}\times[0,\infty)\to\RR$ be as in Example \ref{exa:F} and $\Lam\in \scr G_{[\lam]}^1$.
	There exist constants $c_1,c_2\in(0,\infty)$ (depending on $F$ and $\lambda$) such that for all $\phi\in \DDp$ we have
	\begin{align}\label{eq:AH}
		-W_F(\ll\circ\phi^{-1})  &=-   \int_{\RR^d}F\big(\phi(x),\tfrac{\rr_\ll}{|\det[\na\phi]|}(x)\big) |\det[\na\phi(x)]| \d x \\
		&\leq  c_1 \int_{\RR^d}\Big(1+|\phi(x)|  -\ln \big(\tfrac{\rr_\ll}{|\det[\na\phi]|}(x)\big)  \Big)\rr_\ll(x)\d x\nonumber\\
		&\leq  c_1 \int_{\RR^d}\Big(1+|\phi(x)|  -\ln (\rr_\ll(x))+d\ln\big(|\na\phi(x)|_\textnormal{op}\big)  \Big)\rr_\ll(x)\d x\nonumber\\
		&\leq  c_2 (1+ \|\phi\|_{\D}),\nonumber
	\end{align}
	where $|\na\phi(x)|_\textnormal{op}$ denotes the operator norm of the Jacobian matrix $\na\phi(x)$ and 
	\begin{equation*}
		{\|\phi\|}_{\D}:=\int_{\RR^d}|\phi|+|\nn \phi|_\textnormal{op}\d\ll.
	\end{equation*}
	Since $\|\cdot\|_{\D}$ is a measurable norm on $(\D,d_\D)$, for any Gaussian measure $G_\D$ there exists
	$\alpha>0$ such that $G_\D(\e^{\aa\|\cdot\|_{\D}^2})<\infty$, see \cite[Thm.~2.8.5]{B98}.
	This implies $G_\D(\e^{\|\cdot\|_{\D}})<\infty$ and so,
	\begin{equation*}
		Z_F=\int_{\D_1} \e^{-W_F(\ll\circ\phi^{-1})}G_\D(\d\phi)<\infty
	\end{equation*}
	in view of \eqref{eq:AH} for any Gaussian measure $G_\D$ on $\D$. In particular, the measure $\LL_F(\d\mu):= Z_F^{-1} \e^{-W_F(\mu)}\LL(\d\mu)$ is well-defined, since $\Lam\in \scr G_{[\lam]}^1$.
	By Theorem \ref{thm:main}  there exists a non-terminating diffusion on $\scr P_2$ which is associated with the minimal closed extension of $(\EE^F,C_b^1(\scr P_p))$ as in  \eqref{eq:EF}.
\end{exa}

The diffusion obtained under $(C_3)$ by Theorem \ref{thm:main} is denoted by ${(\mu_t)}_{t\ge 0}$ in the following. It has invariant measure  
$\Lam_F$ as in \eqref{eq:LW}. We can regard ${(\mu_t)}_{t\ge 0}$ as a perturbation of the process 
${(R_t)}_{t\ge 0}$ with invariant measure $\Lam$, which arises from the choice $F=0$ and has Dirichlet form $(\EE,\D(\EE)):=(\EE^{\gamma,\Lam},\D(\EE^{\gamma,\Lam}))$ as in Proposition \ref{prp:basicP}.

To make sense of the drift term involving $D^\gamma W_F$ for an energy functional $W_F\in\D(\EE)$, we introduce the \lq weak intrinsic derivative\rq,
 which is a natural extension of the intrinsic gradient $D$ in Definition \ref{D1}.
By construction of the minimal closed extension in Proposition \ref{prp:basicP} (i),
for any $u\in \D(\EE)$,  we find a sequence  ${(f_m )}_{m}
\subset   C_b^1(X)$, where $X:= L^p(\RR^d\to\R^d,\lam)$, such that: 
\begin{itemize}
	\item[(a)] $f_m\to u\circ\Psi_\lam$  in $L^2(\D,G_\D)$;
	\item[(b)] $f_m=u_m\circ\Psi_\lam$ for  some $u_m\in C_b^1(\scr P_p)$;
	\item[(c)] $\displaystyle \int_\D{\big\la\na (f_m-f_k),\na (f_m-f_k)\big\ra}_{L^2(\R^d\to\R^d,\ll)}\d G_\D\overset{m,k\to\infty}{\longrightarrow}0$, i.e.
	${(\na f_m)}_{m}$ is a Cauchy sequence in $L^2(\D\to L^2(\R^d\to\R^d,\ll),G_\D)$.
\end{itemize}
The limit
\begin{align*}
	\D\ni \phi\mapsto	\tilde D^u(\phi):=\lim_{m\to\infty}\na f_m(\phi)\qquad\text{in }L^2(\D\to L^2(\R^d\to\R^d,\ll),G_\D)
\end{align*}
only depends on $u$  (it does not depend on the choice of an approximating sequence ${(f_m)}_{m}$, as a consequence of the closability) and is measurable w.r.t.~$\sigma(\Psi_\lam)$.
We write $\mu\times \LL(\d\mu)$ for the probability measure on $\R^d\times\scr P_p$ such that
\begin{equation*}
\big(\mu\times \LL(\d\mu)\big)(A)=\int_{\scr P_p}\int_{\RR^d}\eins_A(x,\mu)\mu(\d x)\Lam(\d \mu)
\end{equation*}
for any measurable set $A\subseteq \R^d\times\scr P_p$.
In view of $\LL=G_\D\circ\Psi_\lam^{-1}$ the linear map
\begin{align*}
	L^2\big(\R^d\times\scr P_p\to\RR^d,\mu\times \LL(\d\mu)\big)\,&\longrightarrow\, L^2(\D\to L^2(\R^d\to\R^d,\ll),G_\D)\\
	{\{V(x,\mu)\}}_{x\in\RR^d,\mu\in\scr P_p}&\longmapsto \big\{V\big(\phi(\cdot),\Psi_\lam(\phi)\big)\big\}_{\phi\in\D}
\end{align*}
is an isometry. By (b) and \eqref{eq:ChainRIntrD} we conclude that $\tilde D^u$ lies in the image set of this isometry.
We obtain the following notion of a weak intrinsic gradient.

\beg{defn}[Weak intrinsic gradient]\label{def:WID} For  $u\in  \D(\EE)$ the unique element  
\begin{equation*}
Du\in  L^2\big(\R^d\times\scr P_p\to\RR^d,\mu\times \LL(\d\mu)\big)
\end{equation*}
such that
\begin{equation*} D u\big(\phi(\cdot),\Psi_\lam(\phi)\big)= \tilde D^u(\phi),\qquad G_\D\text{-a.e.~}\phi,\end{equation*}
is called  the \textit{weak intrinsic gradient} of $u$. 
\end{defn}
It is clear that the weak intrinsic gradient of $u\in C_b^1(\scr P_p)$ coincides with the intrinsic derivative $Du$.
Moreover, the square-field operator $\Gamma$ of $(\EE,\D(\EE))$ extends to
\begin{equation}\label{eq:weakSF}
	\Gamma(u,v)(\mu)= \int_{\RR^d}\gamma_\mu(x) \big\la Du(x,\mu), Dv(x,\mu)\big\ra\mu(\d x),\quad \Lam\text{-a.e.~}\mu\in\scr P_p,\,u,v\in\D(\EE).
\end{equation}

We formulate a corollary of Theorem \ref{thm:main}. 
\begin{cor}\label{cor:SFG}
	In addition to $(C_3)$ we assume $W_F\in\D(\EE)$ and $\e^{-W_F}\in\D(\EE)\cap L^\infty(\scr P_p,\Lam)$. Let $(A,\D(A))$ be the generator of $(\EE,\D(\EE))$ in $L^2(\scr P_p,\Lam)$.
	The diffusion $(\Omega, \scr F,(\mu_t)_{t\geq 0},$ $(\P^\mu)_{\mu\in \scr P_p})$ on $\scr P_p $ 
	which is  associated with $(\EE^F,\D(\EE^F))$  yields a solution to  \eqref{eq:M}. More precisely:
	
	Let $u\in \D(A)$ and $\tilde u$ denote an $\EE^F$-quasi-continuous version. Then,
	\begin{equation}\label{eq:AF}
		\tilde u(\mu_t)-\tilde u(\mu_0)-\int_0^t\Big[Au(\mu_s)-\int_{\RR^d}\big\la\gamma_{\mu_s}(x) DW_F(x,\mu_s),Du(x,\mu_s)\big\ra\mu_s(dx)\Big]\d s,\qquad t\ge 0,
	\end{equation}
	is a $\P^\mu$-martingale for $\EE^F$-q.e.~$\mu\in\scr P_p$.
\end{cor}
\begin{proof}
	We note that $\e^{-W_F}\in L^\infty(\scr P_p,\Lam)$ implies $\D(A)\subseteq \D(\EE)\subseteq \D(\EE^F)$. 
	Due to integrability of $Au(\mu)-\int_{\RR^d}\la\gamma_{\mu}(x) DW_F(x,\mu),Du(x,\mu)\ra\mu(dx)$ w.r.t.~$\Lam_F(\d\mu)$ for $u\in\D(\EE^F)$,
	\eqref{eq:AF} is well-defined as an additive functional of the process ${(\mu_t)}_{t\ge 0}$.
	We can apply Lemma \ref{lem:generator} with the choice $\varrho:=\e^{-W_F}$, using  $D\varrho=-\e^{-W_F}DW_F$.
	The statement follows by Remark \ref{rem:MS}.
\end{proof}

\section{Functionals with local weak gradient}\label{sec:Loc}
This section is motivated by the following observation.
In case of Example \ref{exa:F2} we cannot show that $W_F$ has a weak intrinsic gradient in the sense of Definition \ref{def:WID}, because the candidate for its gradient $D W_F(\mu)= \na V+\tfrac{q(\rho_\mu)\na\rho_\mu}{\rho_\mu}$, where $\rho_\mu=\frac{\d\mu}{\d x}$, doesn't have the desired integrability properties w.r.t.~$\Lam(\d\mu)$. 
We introduce the notion of a \lq local weak intrinsic gradient\rq.
Then, under suitable Lipschitz conditions on $F$ and $\partial_2F$, by Theorem \ref{thm:localDom} below, $W_F$ has a local gradient in this sense. The application to Example \ref{exa:F2} is addressed subsequently.

Let $\lam$, $G_\D$ satisfy $(C_1)$ (see Section \ref{ssec:indT}) with $(\D,d_\D)$ as in \eqref{eq:DD1}. Defining
\begin{equation*}
	\D_2:=\big\{\phi\in \D:\na \phi\in C_b^1(\RR^d,\RR^{d\times d}) \big\}
\end{equation*}
and $\DDp$ as in \eqref{eq:DDp}, throughout this section we additionally assume
\begin{equation*}
G_\D(\D_1\cap\D_2)=1
\end{equation*}
and the full topological support of $G_\D$ on $\D_2$ (i.e.~$G_\D$ is non-zero for any non-empty, open subset of $\D_2$ w.r.t.~$d_{\D_2}$ in \eqref{eq:dDD2} below).
First, we argue why we can find a reasonable measure (in fact a restricted Gaussian) with that assumption.
The linear space $\D_2$ is complete regarding the metric
\begin{equation}
	\label{eq:dDD2} d_{\D_2}(\phi_1,\phi_2):=d_\D(\phi_1,\phi_2)  +\big\| \na^2\phi_1-\na^2\phi_2 \big\|_{\infty },\quad \phi_1,\phi_2\in\D_2,
\end{equation}
where we write
\begin{equation*}
\|\na^2\phi\|_\infty:=\sup_{\substack{x,y\in\RR^d\\x\ne y}}\ff{|\na\phi(x)-\na\phi(y)|_{\textnormal{op}}}{|x-y|}\end{equation*}
for $\phi\in\D_2$ and $|\cdot|_{\textnormal{op}}$ denotes the operator norm for matrices.

\begin{rem}\label{rem:Gauss2} 
	Let  $\lam\in\ppac$, $p\in[1,2]$.
A choice for $G_\D$ which fulfills $G_\D(\D_1\cap\D_2)=1$, $(C_1)$ and has full topological support on $\D_2$ can always be found in the class of restricted Gaussian measures, as follows. First, we define a non-degenerate Gaussian measure 
on the vector space $\scr D_2$, analogously as done in Example \ref{exa:GaussD} for the space $\D$.
This is achieved by 
starting with any sequence ${\{\phi_k\}}_{k\in\mathbb N}\subset\D_2$ which is an
orthonormal basis of $L^2(\RR^d\to\RR^d,\lam)$
and replacing the choice of $a_k$ in \eqref{eq:ak} by
\begin{equation*}
	a_k:=\max\big\{2^k\|\na\phi_k\|^2_\infty,2^k \|\na^2\phi_k\|^2_\infty, 1\big\},\qquad k\in\NN.
\end{equation*}
Then, we proceed in the exact same way and obtain a non-degenerate Gaussian measure on $\scr D_2$ (the estimates are analog, but carried out for $\|\na ^2 \phi \|_{\infty }$ as well).
Now, Property $(C_1)$ is preserved if we restrict that Gaussian to the set $\scr D_1$, which is open in $(\D,d_\D)$, see Remarks \ref{rem:loc} \&  \ref{rem:DDnew}.
So, by restriction to $\scr D_1$ (and re-scaling)  we end up with a measure $G_\D$ on $\scr D_2$ such that $G_\D(\D_1\cap\D_2)=1$ and Condition $(C_1)$ holds true. 
\end{rem}

To define the local weak intrinsic gradient we introduce the sets
\begin{equation*}
	\D^{(n)}:=\Big\{\phi\in \D_2\cap\DDp: \|\nn(\phi^{-1})\|_\infty
	+ d_{\D_2}(\phi,0)<n\Big\},\qquad n\in\NN,
\end{equation*}
which are increasing  to $\DDp\cap\D_2$ as $n\uparrow\infty$ and 
each $\D^{(n)}$ is open (see Remark \ref{rem:DDDnew} below) in $(\D_2,d_{\D_2})$. 
For $n\geq 3$ the set $\D^{(n)}$ is non-empty (it contains the identity map) and we define
\begin{equation}\label{eq:LamEps} \Lam^{(n)}:=\Big(\frac{\eins_{\D^{(n)}}G_\D}{G_\D(\D^{(n)})}\Big)\circ\Psi_\lam^{-1}
\qquad\text{and}\qquad \Lam:=G_\D\circ\Psi_\lam^{-1}.\end{equation}

\begin{rem}\label{rem:DDDnew}
	$\D_1$ is open in  $(\D_2,d_{\D_2})$ as a consequence of Remark \ref{rem:DDnew}.
	To show that $\D^{(n)}$ is open in $(\D_2,d_{\D_2})$ it suffices to prove that the set
	\begin{equation*}
		A_M:=\big\{\phi\in \D_2: \|\nn(\phi^{-1})\|_\infty<M\big\}
	\end{equation*}
	is open for each $M\in(0,\infty)$. Let $\phi_0\in A_M$, $\varepsilon<1$ such that $\|\nn(\phi_0^{-1})\|_\infty< \varepsilon M$
	 and $\phi\in\D_2$ such that 
	\begin{equation*} d_{\D_2}(\phi,\phi_0)<\|\nn(\phi_0^{-1})\|_\infty^{-1}-(\varepsilon M)^{-1}.\end{equation*}
	Then,
	\begin{align*}
		\frac{|x-y|}{\varepsilon M}&< \frac{|x-y|}{\|\nn(\phi_0^{-1})\|_\infty}-\|\na(\phi-\phi_0)\|_\infty\\
		&\leq |\phi_0(x)-\phi_0(y)|-\big|(\phi-\phi_0)(x)-(\phi-\phi_0)(y)\big|\\
		&\leq |\phi(x)-\phi(y)|
	\end{align*}
	and hence  $\|\nn(\phi^{-1})\|_\infty\leq \varepsilon M<M$, i.e.~$\phi\in A_M$.
\end{rem}

By Proposition \ref{prp:basicP}, we obtain a Dirichlet form $(\EE^{(n)},\D(\EE^{(n)})):=(\EE^{\gamma,\Lam^{(n)}},\D(\EE^{\gamma,\Lam^{(n)}}))$ in 
$L^2(\scr P_p ,\LL^{(n)})$  for each $n\in\NN_{\geq 3}$, as well as $(\EE,\D(\EE)):= (\EE^{\gamma,\Lam},\D(\EE^{\gamma,\Lam}))$.
 The domains $\D(\EE^{(n)})$ are decreasing in $n$ and
$$\D(\EE)\subseteq  \D(\EE^{(\infty)}) :=\bigcap_{n=3}^\infty\D(\EE^{(n)}).$$
It should be mentioned that  $\supp[\LL^{(n)}]$ is compact w.r.t.~the $p$-Wasserstein topology, due to
continuity of the map $$L^p(\RR^d\to\R^d,\lam)\ni\phi\mapsto\Psi_\ll(\phi)=\lam\circ\phi^{-1}\in \scr P_p $$ and
the fact that bounded sets in $(\D_2,d_{\D_2})$ are precompact in $L^p(\RR^d\to\R^d,\lam)$.

Definition \ref{def:LIWD} below is analogue to Definition \ref{def:WID}, but using the notion of the weak gradient coming from the family of Dirichlet forms ${\{\EE^{(n)}\}}_n$ instead of $\EE$.
By construction of the minimal closed extension in Proposition \ref{prp:basicP} (i),
for any $u\in \D(\EE^{(n)})$,  we find a sequence  ${(f_m )}_{m}
\subset   C_b^1(X)$, where $X:= L^p(\RR^d\to\R^d,\lam)$, such that
the analogues of (a), (b), (c) at the end of Section \ref{sec:EF} hold true regarding $\eins_{\D^{(n)}} G_\D$.
The limit
\begin{align*}
	\D\ni \phi\mapsto	\tilde D^u(\phi):=\lim_{m\to\infty}\na f_m(\phi)\qquad\text{in }L^2(\D\to L^2(\R^d\to\R^d,\ll),\eins_{\D^{(n)}}G_\D)
\end{align*}
only depends on $u$  (it does not depend on the choice of an approximating sequence ${(f_m)}_{m}$, as a consequence of the closability) and is measurable w.r.t.~$\sigma(\Psi_\lam)$.
Moreover, by the hierarchy of open sets,
\begin{equation*}
	\D^{(n_1)}\subset\D^{(n_2)}\quad\text{for}\quad  n_1\le n_2,
\end{equation*}
we conclude that  for $G_\D$-a.e.~$\phi\in\D_1\cap\D_2$ the element  $\tilde D^u(\phi)$ does not depend on
a particular choice of $n$ with $\phi\in \D^{(n)}$, given that $u \in\D(\EE^{(\infty)})$.
So, \begin{equation*}\tilde D^u\in \bigcap_{n=3}^\infty L^2\big(\D\to L^2(\R^d\to\R^d,\ll),\eins_{\D^{(n)}}G_\D\big)\end{equation*} is well-defined for $u \in\D(\EE^{(\infty)})$.  
Based on the above discussion, we introduce the following notion of a local weak intrinsic derivative (cf.~Def.~\ref{def:WID}):

\beg{defn}[Local weak intrinsic gradient]\label{def:LIWD} For  $u\in  \D(\EE^{(\infty)})$ the unique element  
\begin{equation*} Du\in \bigcap_{n=3}^\infty L^2\big(\R^d\times\scr P_p\to\RR^d,\mu\times \LL^{(n)}(\d\mu)\big)\end{equation*} such that
\begin{equation*} D u\big(\phi(\cdot),\Psi_\lam(\phi)\big)= \tilde D^u(\phi),\qquad G_\D\text{-a.e.~}\phi,\end{equation*} 
is called the \textit{local weak intrinsic gradient} of $u$. 
\end{defn}
It is clear that for $u\in C_b^1(\scr P_p)$ (respectively $u\in \D(\EE)$) the local weak intrinsic gradient coincides with the (weak) intrinsic gradient in Def.~\ref{D1} (Def.~\ref{def:WID}).

In the following theorem, we denote by $\nn_1$ and $\pp_2$ the
gradient w.r.t.~$x$, respectively the derivative w.r.t.~$s$, for a function with argument $(x,s)\in \R^d\times (0,\infty)$.

\begin{thm}\label{thm:localDom} We assume $(C_1)$, $(C_2)$ (see Sections \ref{ssec:indT} and \ref{sec:EF}) with $G_\D(\D_2\cap \D_1)=1 $  and $\ll\in \ppac$ such that $\rho_\lam$ is strictly positive and Lipschitz continuous. Then:

(i) $\rho_\mu$ is Lipschitz continuous for $\Lam$-a.e.~$\mu$.

(ii) If $F$ and $\pp_2F$ are locally Lipschitz continuous on $\R^d\times (0,\infty)$ and
\begin{equation*}
	\int_{\D}\big\||(\na_1\partial_2 F)(\cdot,\rho_\mu)|+\big|(\partial_2\partial_2 F)(\cdot,\rho_\mu)\na\rho_\mu\big|
	+|(\partial_2 F)(\cdot,\rho_\mu)|\big\|^2_{L^2(\RR^d,\mu)}\Lam^{(n)}(\d\mu)<\infty
\end{equation*}
for $n\in\NN_{\geq 3}$, then the functional $W_F$ defined in \eqref{eq:W} has a local weak intrinsic gradient
given by
\begin{equation}\label{eq:DechiF}
	D W_F(\mu)(x)= H_F(x,\mu):=(\na_1\partial_2 F)(x,\rho_\mu(x))+(\partial_2\partial_2 F)(x,\rho_\mu(x))\na\rho_\mu(x)
\end{equation}
in $\mu\times \LL(\d\mu)$-a.e.~sense on $\R^d\times \scr P_p$.
\end{thm}

\begin{proof} We complete the proof by four steps. Claim (i) follows from Step (1). Claim (ii) is verified in Steps (2)-(4). 
	
	(1) First, we show that $\rr_\mu$ is Lipschitz continuous for $\LL^{(n)}$-a.e.~$\mu$ and any $n\in\NN_{\geq 3}$, with a Lipschitz constant only depending on $n$.  We recall
	\beq\label{CR1}
	\rr_{\mu}=\Big(\frac{\rho_\lam}{|\det\na\phi|} \Big)\circ\phi^{-1}\qquad \text{for }\phi\in\DDp,\ \mu=\Psi_\ll(\phi)=\ll\circ\phi^{-1},
\end{equation}
see Remark \ref{rem:rhomu}.
For any $\phi\in \D^{(n)}$, we have
\beq\label{CR2} \max\Big\{\|\na\phi\|_\infty,\, \|\nn^2\phi\|_\infty,\, \|\na (\phi^{-1})\|_\infty,\, |\phi(0)|\Big\}\le n.\end{equation}
Then
\begin{equation*}
|\phi^{-1}(0)|\leq n +|\phi^{-1}(0)-\phi(0)|\leq n+ n|\phi(\phi(0))|\le
n+n|\phi(0)|+n\|\nn\phi\|_\infty|\phi(0)|\leq 3n^3
\end{equation*}
and hence
\begin{equation*}
|\phi^{-1}(x)|\leq 3 n^3+   \|\nn(\phi^{-1})\|_\infty|x|\le 3 n^3+n|x|.
\end{equation*}
This together with \eqref{CR2} yields
\begin{align}\label{eq:uniEst1}
&\inf_{|x|\le r} \Big(\frac{\rho_\lam}{|\det\na\phi|}\Big)(\phi^{-1}(x))
\geq n^{-d}  \inf_{|x|\le 3n^3+nr}\rho_\lam(x)> 0,\\
&\sup_{|x|\le r} \Big(\frac{\rho_\lam}{|\det\na\phi|}\Big)(\phi^{-1}(x))
\le n^d  \sup_{|x|\le 3n^3+nr}\rho_\lam(x)<\infty, \qquad r\in(0,\infty).\nonumber
\end{align}
Moreover, for $\mu:=\Psi_\lam(\phi)$, $\phi\in\D^{(n)}$, \eqref{CR1} implies
\begin{equation}\label{eq:naphi} \na\rho_\mu=\Bigg(\frac{(\na\phi)^{-1}\na\rho_\lam}{|\det\na\phi|}
-{\rm sgn}(\det[\nn\phi])\frac{\rho_\lam(\na\phi)^{-1}(\na\det\na\phi)}{|\det\na\phi|^2}\Bigg)\circ\phi^{-1}
\end{equation}
and thus by \eqref{CR2} we find a constant $c_n\in (0,\infty)$   such that
\begin{equation}\label{eq:lipalph}
\|\nn\rho_\mu\|_\infty\le c_n,\ \  \mu\in\Psi_\lam(\D^{(n)}). \end{equation}
In particular, \eqref{eq:lipalph} holds for $\LL^{(n)}\text{-a.e.}~\mu$.

(2) In this step we calculate the local weak intrinsic gradient $DW_F$ for $F$ additionally satisfying
\begin{equation}\label{FW}
	F\in C^2(\R^d\times (0,\infty)),\quad \bigcup_{s>0}{\rm supp} [F(\cdot,s)] \subset [-l,l]^d\ \text{for\ some \   } l>0.\end{equation}
The step involves approximating $W_F$ by a suitable sequence in $C_b^1(\scr P_p)$ for each $n$, as $\D(\EE^{(n)})$ is the closure of $C_b^1(\scr P_p)$
w.r.t.~$(\EE^{(n)}_1)^{1 /2 }$-norm.
As $W_F$ is not defined everywhere on $\scr P_p$, our proof includes mollifying techniques.
We choose a mollifier $  \tau\in  C^\infty(\RR^d)$ with compact support such that $\tau(x)=1$ if $|x|\le\vv$ for some constant $\vv>0$, 
\begin{equation*}
0\le \tau(\cdot)\le 1\qquad\text{and}\qquad \int_{\R^d}\tau(x)\d x=1.
\end{equation*}
For any $m\in\NN$ and $ x\in\R^d$,   let $\tau_m(x):= m^d \tau(mx)$. Define
\beg{equation*} 
(\tau_m*\mu)  (x):= \int_{\RR^d}\tau_m(x-y)  \mu(\d y),\ \mu\in\scr P,\, x\in\R^d.
\end{equation*}
We have $\tau_m*\mu \in C_b^\infty(\R^d)$ and from \eqref{eq:uniEst1} \& \eqref{eq:lipalph}   it follows
\begin{align}\label{NB1}& \inf_{\mu\in\Psi_\lam(\D^{(n)})}\,\inf_{[-l,l]^d}\,\tau_m*\mu>0\ \ 
\qquad\text{for }m\in\N,\\
&\lim_{m\to\infty} \sup_{\mu\in \Psi_\ll(\D^{(n)})} \big(\|\rr_\mu- \tau_m*\mu \|_\infty+\|\nn\rr_\mu
-\nn (\tau_m*\mu)\|_{L^2(\R^d\to\RR^d,\mu)}\big)=0,\ \ n\in\mathbb N.\nonumber\end{align}
Moreover, in view of \eqref{FW} \&\eqref{NB1}, there exists a function $\tilde F\in C^2(\RR^{d+1})$ with compact support ($\tilde F$ only depends on $F$ and $n$) such that 
$$F(x, (\tau_m*\mu)(x))=\tilde F(x, (\tau_m*\mu)(x))$$ 
for $x\in\RR^d,\,m\in\N,\,\mu\in \Psi_\ll(\D^{(n)})$.
Defining  
\begin{equation*}
u_{m}(\mu):=\int_{\RR^d}\tilde F\big(x, (\tau_m*\mu)(x)\big)\d x,\qquad \mu\in\scr P_p,
\end{equation*}
 we have $u_m\in C_b^1(\scr P_p)$ and also
\begin{align*}
 u_{m}(\mu)&=\int_{\RR^d}F\big(x, (\tau_m*\mu)(x)\big)\d x,\\
 D  u_m(x,\mu)&=\int_{\R^d} \big[\pp_2 F(y,(\tau_m*\mu)(y))\big](-\nn\tau_m)(y-x)\d y
\end{align*}
for $\mu\in \Psi_\ll(\D^{(n)})$.
With integration by parts we get
$$D u_m(x,\mu)= \int_{\R^d} \big[(\nn_1\pp_2 F)(\cdot,\tau_m*\mu) + (\pp_2\pp_2 F)(\cdot,\tau_m*\mu) 
\nn (\tau_m*\mu)\big](y)\tau_m(y-x)\d y.$$
Let $H_F$ be defined as in the statement of this theorem. Using \eqref{FW} and \eqref{NB1}, we obtain
\begin{equation*}
\lim_{m\to\infty} \sup_{\mu\in \Psi_\ll(\D^{(n)})} \Big(|u_m(\mu)-W_F(\mu)|
+ \big\|Du_m(\cdot,\mu)- H_F(\cdot,\mu)\big\|_{L^2(\R^d\to\RR^d,\mu)}\Big)=0,\qquad n\in\NN.
\end{equation*}
This convergence is stronger than convergence w.r.t.~$(\EE^{(n)}_1)^{1 /2 }$-norm.
Hence, $W_F\in \D(\EE^{(n)})$ for $n\geq 3$ and $W_F$ has local weak intrinsic derivative $DW_F= H_F$ as claimed.

(3) We calculate the local weak intrinsic derivative $DW_F$ for $F$ as in the assumptions, but keep the support condition
\begin{equation}\label{eq:supp2}
	\bigcup_{s>0}{\rm supp} [F(\cdot,s)] \subset [-l,l]^d\ \text{
		for\ some \   } l>0\end{equation}
	of the previous step.	By mollifying $F$ we can reduce the problem to step (2).
	We choose  $\kappa\in  C^2(\RR)$ with compact support in $(-1,1)$ such that $\kappa(s)=1$ if $|s|\le\vv$ for some constant $\vv>0$, 
	\begin{equation*}
		0\le \kappa(\cdot)\le 1\qquad\text{and}\qquad \int_{\R}\kappa(s)\d s=1.
	\end{equation*}
For $m \geq 1$ and $(x,s)\in \R^d\times (0,\infty)$ let
\beg{align*} 
F_m(x,s):=m\int_s^{s+2m^{-1}} \kappa(mr-1)\int_{\R^d} F(y, r) \tau_m(x-y)\d x\d r.\end{align*}
Then $F_m$ satisfies \eqref{FW}. Hence, $W_{F_m}$ has 
local weak intrinsic derivative $D W_{F_m}= H_{F_m}$ by step (2). From \eqref{eq:uniEst1}, \eqref{eq:lipalph} \& \eqref{eq:supp2} together with the local Lipschitz assumptions for $F$ and $\partial_2 F$ we conclude
\begin{equation*}
\sup_{m\geq 1}\sup_{\mu\in\Psi_\lam(\D^{(n)})}\big\||W_{F_m}(\mu)|+|H_{F_m}(\cdot,\mu)|\big\|_{L^\infty(\RR^d,\mu)}<\infty,\qquad n\in\NN.
\end{equation*}
Consequently, for each $n$,
\beq\label{NNd}\lim_{m\to\infty} \Big(\big\|W_{F_m}-W_F\big\|_{L^2(\scr P_2,\LL^{(n)})} + \big\|
DW_{F_m}-H_{F}\big\|_{L^2(\R^d\times \scr P_2\to\RR^d, \mu\times\LL^{(n)}(\d\mu))}\Big)=0.\end{equation}
This means convergence w.r.t.~$(\EE^{(n)}_1)^{1 /2 }$-norm and
so, $W_F$ has local weak intrinsic gradient $D W_F= H_F.$

(4)  Let $F$ be as in the assumptions of the theorem. For any $m\geq 1$, let
$$F_m(x,s):= \tau(m^{-1}x) F(x,s).$$
Then $F_m$ satisfies Condition \eqref{eq:supp2} and step (3) implies that $W_{F_m}$ has local weak 
intrinsic derivative $D W_{F_m}=H_{F_m}$. Since 
\begin{equation*}
	|H_{F_m}(\cdot,\mu)|\leq c_\tau\Big(\big|(\na_1\partial_2 F)(\cdot,\rho_\mu)\big|+\big|(\partial_2\partial_2 F)(\cdot,\rho_\mu)\na\rho_\mu\big|
	+\big|(\partial_2 F)(\cdot,\rho_\mu)\big|\Big)
\end{equation*}
for $m\geq 1$ and some constant $c_\tau$ depending only on $\tau$, we obtain \eqref{NNd}
for each $n$ by the integrability assumption for the right-hand side in the above estimate and Lebesgue's dominated convergence.
So, $W_F$ has local weak intrinsic derivative $D W_F= H_F.$ 
\end{proof} 

We come back to Examples \ref{exa:F} \& \ref{exa:F2} now and let $F,V,q$ be as given there.
We choose $p=2$, $\lam\in\twopac$ with $\rho_\lam$ being strictly positive, Lipschitz continuous such that 
\begin{equation}\label{eq:logderlam}
	\int_{\RR^d}\Big(|\ln(\rho_\lam)|+\tfrac{|\na\rho_\lam|}{\rho_\lam}\Big)^2\d\lam<\infty.
\end{equation}
For example, $\lam$ could be the distribution of a Gaussian random variable on $\RR^d$.
By \eqref{eq:logderlam} we ensure that Theorem \ref{thm:localDom} (ii) is applicable in this setting.
Moreover, we choose $G_\D$  as in Remark \ref{rem:Gauss2}.

We define $\Lam, \Lam^{(n)}$ as in \eqref{eq:LamEps} for $n\in\NN_{\geq 3}$.
Up to normalization constants, we have $\LL^{(n)}\uparrow \LL$.
Let $$\Lam^{(n)}_F:=\frac{\e^{-W_F}\d\Lam^{(n)}}{\Lam^{(n)}(\e^{-W_F})}$$ analogous to \eqref{eq:LW}, \eqref{eq:W} and ${(\mu_t^{(n)})}_{t\geq 0}$ denote the diffusion on $\scr P_2$ 
with Dirichlet form $(\EE^{F,n},\D(\EE^{F,n})):=(\EE^{\gamma,\Lam_F^{(n)}},\D(\EE^{\gamma,\Lam_F^{(n)}}))$ as defined in Theorem \ref{thm:main}.
The corresponding objects for the undisturbed case $F=0$ are denoted by ${(R_t^{(n)})}_{t\ge 0}$ and $(\EE^{(n)},\D(\EE^{(n)})):=(\EE^{\gamma,\Lam^{(n)}},\D(\EE^{\gamma,\Lam^{(n)}}))$,
as above in this section.

Combining the results of Sections \ref{sec:EF} and Theorem \ref{thm:localDom}, ${(\mu_t^{(n)})}_{t\geq 0}$ can be interpreted as intrinsic stochastic gradient flow on $\scr P_2$ formally satisfying
\begin{equation}\label{eq:SGF2}
	\d\mu^{(n)}_t=-D^\gamma W_F(\mu^{(n)}_t)\d t+\d R^{(n)}_t,\qquad t\ge 0,
\end{equation}
where
\begin{equation}\label{eq:DWF}
	DW_F(\cdot,\mu)=\na V+\frac{q(\rho_\mu)\na\rho_\mu}{\rho_\mu}
\end{equation}
is the local weak intrinsic gradient of $W_F$ for $F$ as in Examples \ref{exa:F} \& \ref{exa:F2}.
 The process ${(\mu^{(n)}_t)}_{t\ge 0}$ lives on the compact set $E_n:=\supp[\Lam^{(n)}]\subset \twopac$. The sets $E_n$ are increasing in $n$ and
 $$\bigcup_{n\geq 3}E_n\subseteq \scr P_2 \quad\text{densely w.r.t.~}\W_2.$$
 The precise formulation of \eqref{eq:SGF2} involves the generator $A^{(n)}$ of $\EE^{(n)}$, analogous to \eqref{eq:M} and Corollary \ref{cor:SFG}.
\begin{cor}\label{cor:locSGH}
	Let $n\in\NN_{\geq 3}$. The diffusion $((\mu^{(n)}_t)_{t\geq 0},$ $(\P^{\mu,n})_{\mu\in E_n})$ on $E_n$ 
	which is  associated with $(\EE^{F,n},\D(\EE^{F,n}))$  yields a solution to \eqref{eq:SGF2}, \eqref{eq:DWF} in the following sense:
	
	For $u\in \D(A^{(n)})$,
	\begin{equation*}
		\tilde u(\mu_t^{(n)})-\tilde u(\mu_0)-\int_0^t\Big[A^{(n)}u(\mu_s^{(n)})-\int_{\RR^d}\gamma_{\mu_s^{(n)}}(x)\big\la DW_F(x,\mu_s^{(n)}),Du(x,\mu_s^{(n)})\big\ra\mu_s^{(n)}(\d x)\Big]\d s,\quad t\ge 0,
	\end{equation*}
	is a $\P^{\mu,n}$-martingale for $\EE^{F,n}$-q.e.~$\mu\in E_n$, where  $\tilde u$ denotes an $\EE^{F,n}$-quasi-continuous version.
\end{cor}
\begin{proof}
	The statement follows applying Theorem \ref{thm:localDom} and then Corollary \ref{cor:SFG} regarding the localized objects. 
	To verify the assumptions we compute 
	\begin{equation*}
		\partial_2 F(x,s)=V(x)+\int_1^s\frac{q(r)}{r}\d r,\quad \nabla_1\partial_2 F(x,s)=\nabla V(x),\quad \partial_2\partial_2 F(x,s)=\frac{q(s)}{s}.
	\end{equation*}
	For $n\in\N_{\geq 3}$ and $\mu\in\Psi_\lam(\D^{(n)})$ (in view of \eqref{eq:uniEst1}, \eqref{eq:lipalph}, \eqref{CR1} \& \eqref{eq:naphi}) there exists a constant $c_n\in(0,\infty)$ such that
	\begin{align*}
		\|\partial_2 F(\cdot,\rho_\mu(\cdot))\|_{L^2(\RR^d,\mu)}&\leq c_n \big\|1+|\cdot|+|\ln(\rho_\lam)|\big\|_{L^2(\RR^d,\lam)},\\
		\|\nabla_1\partial_2 F(\cdot,\rho_\mu(\cdot))\|_{L^2(\RR^d\to\RR^d,\mu)}&\leq c_n,\\
		\|\partial_2\partial_2 F(\cdot,\rho_\mu(\cdot))\nabla\rho_\mu\|_{L^2(\RR^d,\mu)}&\leq c_n\big\|\tfrac{\nabla\rho_\lam}{\rho_\lam}\big\|_{L^2(\RR^d\to\RR^d,\lam)}.
	\end{align*}
	Hence, we have $\e^{-W_F}\in D(\EE^{(n)})\cap L^\infty(\Lam^{(n)})$ and $D\e^{-W_F}=-\e^{-W_F}DW_F$.
\end{proof}

\begin{rem}\label{rem:last}
	Corollary \ref{cor:locSGH} allows us to choose $F$ as in Remark \ref{rem:gPME} setting $V:=\Phi$ and $q:=\frac{\beta'}{b}$. As pointed out the introduction, setting $\gamma(x,\mu):=b(\rho_\mu)(x)$, the
	deterministic counterpart of \eqref{eq:SGF2} yields a solution to the generalized porous media equation \eqref{eq:gPME}.
\end{rem}

\end{document}